\documentclass[10pt,a4paper,twoside]{article}
\usepackage{amsmath,amssymb,enumerate,theorem}
\usepackage{dsfont}
\usepackage{epsfig,color,graphics,graphicx}
\usepackage{amssymb,amsbsy,amsmath,amsfonts,amssymb,amscd}
\usepackage[english]{babel}
\newcommand{\R}{{\mathbb{R}}}

\newenvironment{proof}{\noindent\textbf{Proof.\ }}{\hspace*{\fill}$\Box$\medskip}
\newtheorem{lemma}{Lemma}[section]
\newtheorem{theorem}{Theorem}
\newtheorem{proposition}[lemma]{Proposition}

\newtheorem{remark}{Remark}

\begin{document}

\title{On the local exact controllability of micropolar fluids with few controls}
\author{Sergio Guerrero\thanks{Laboratoire Jacques-Louis Lions, Universit\'e Pierre et Marie Curie, 75252 Paris C\'edex 05, France. \newline e-mail : guerrero@ann.jussieu.fr.}  \& Pierre Cornilleau\thanks{Teacher at Lyc\'ee Louis-le-Grand, 123, rue Saint-Jacques, 75005 Paris, France.\newline e-mail: pierre.cornilleau@ens-lyon.org.}}
\maketitle

\begin{abstract}
In this paper, we study the local exact controllability to special trajectories of the micropolar fluid systems in dimension $d=2$ and $d=3$.
We show that controllability is possible acting only on one velocity.
\end{abstract}

\tableofcontents
\section*{Introduction}
Let $d \in \{2, 3\}$ and $\Omega \subset \R^d$ be a bounded connected open set whose boundary is regular.
In this paper, we focus on the controllability properties of the so-called micropolar fluids (see the monograph \cite{L}).
In this framework, the fluid velocity field $y\ (=y(t,x) \in \R^d)$ and the angular velocity $\omega\ (=\omega(t,x)\in \R$ if $d=2$ or $\R^3$ if $d=3$)
are driven by the following nonlinear system:
\begin{equation}\label{system} \left\{ \begin{array}{ccc}
     y_t - \Delta y+(y \cdot \nabla)y+\nabla p &=& P_1 \omega+ \mathds{1}_{\cal O} u \\
     \omega_t-\Delta \omega-(d-2)\nabla(\nabla \cdot \omega)+(y \cdot \nabla)\omega&=&\nabla\times y+ \mathds{1}_{\cal O} v\\
    \nabla \cdot y &=& 0\\
     y &=& 0\\
     \omega&=&0\\
     y(0,\cdot)&=&y_0 \\
     \omega(0,\cdot)&=&\omega_0
   \end{array} \right. \begin{array}{c}
     \text{in } Q,\\
     \text{in } Q,\\
     \text{in } Q,\\
     \text{on } \Sigma,\\
     \text{on } \Sigma,\\
     \text{in } \Omega,\\
     \text{in } \Omega,
   \end{array} \end{equation}
where $Q:= (0,T)\times\Omega$, $\Sigma:=(0,T)\times \partial\Omega$, $y_0$ and $\omega_0$ are the velocity and angular velocity at time $t=0$ and $\nabla\times :\mathbb{R}^d\rightarrow\mathbb{R}^{2d-3}$ is the usual curl operator. In this system, we have denoted
$$
P_1\omega:=\left\{\begin{array}{ll}
(\partial_2\omega,-\partial_1\omega)&\hbox{ if }d=2,
\\ \noalign{\medskip}
\nabla\times\omega&\hbox{ if }d=3.
\end{array}
\right.
$$
Moreover, ${\cal O}$ is a nonempty open subset of $\Omega$ called the {\it control domain} and $u$ and $v$ stand for control functions which act over the system during the time $T>0$.
As usual in the context of incompressible fluids, the following vector spaces will be used along the paper :
\begin{equation}\label{H}
H=\left\{ w\in L^2(\Omega): \ \nabla \cdot w=0 \text{ in } \Omega, \ w\cdot\nu=0 \text{ on } \partial\Omega\right\}
\end{equation}
and
$$
V=\left\{w \in H^1_0(\Omega): \ \nabla \cdot w=0 \text{ in } \Omega\right\}.
$$
Here, we have denoted $\nu$ the outward unit normal vector to $\Omega$.

The main question we address in this paper is whether system \eqref{system} is {\it locally exactly controllable to the trajectories with the sole control $v$} (with $u=0$) or  {\it with the sole control $u$} (with $v=0$).


We will call a {\it trajectory} associated to system (\ref{system}) any triplet $(\overline{y}, \overline{p}, \overline{\omega})$ satisfying the system without controls, that is to say :
\begin{equation}\label{freesystem} \left\{ \begin{array}{ccc}
     \overline{y}_t - \Delta \overline{y}+(\overline{y}.\nabla)\overline{y}+\nabla \overline{p} &=& P_1 \overline{\omega}\\
     \overline{\omega}_t-\Delta \overline{\omega}-(d-2)\nabla(\nabla \cdot \overline{\omega})+(\overline{y}.\nabla)\overline{\omega}&=&\nabla \times \overline{y}\\
    \nabla \cdot \overline{y} &=& 0\\
     \overline{y} &=& 0\\
     \overline{\omega}&=&0\\
     \overline{y}(0,\cdot)&=&\overline y_0\\
     \overline{\omega}(0,\cdot)&=&\overline\omega_0
   \end{array} \right. \begin{array}{c}
     \text{in } Q,\\
     \text{in } Q,\\
     \text{in } Q,\\
     \text{on } \Sigma,\\
     \text{on } \Sigma,\\
     \text{in } \Omega,\\
     \text{in } \Omega,
   \end{array} \end{equation}
for some initial data $(\overline y_0,\overline\omega_0)$. In this paper, we are interested in the case where $\overline y\equiv  0$ and we assume that there exists a trajectory  $(0,\overline p,\overline\omega)$ solution of (\ref{freesystem}) such that
\begin{equation}\label{overlineomega}
\overline\omega\in L^2(0,T;H^3(\Omega))\cap H^1(0,T;H^1_0(\Omega)) \text{ and } \overline p\in L^2(0,T;H^3(\Omega))\cap H^1(0,T;H^1(\Omega)).
\end{equation}

\begin{remark}
 Observe that for some $\overline\omega_0$ there exists a nontrivial solution $(0,0,\overline\omega)$ to (\ref{freesystem}) with $\overline y_0=0$. This comes from the fact that there exists nonzero solutions of the spectral problem
  \begin{equation}\label{spectral2}
  \left\{\begin{array}{ccc}\medskip\displaystyle
      -2\Delta z &=& \mu z\\ \displaystyle
      \frac{\partial z}{\partial\nu}&=&0\\
    \end{array} \right. \begin{array}{c}\medskip
      \text{in } \Omega,\\
      \text{    on } \partial\Omega,
    \end{array}
  \end{equation}
  when $\Omega$ is a ball: indeed, one can choose a radially symmetric function $z$ satisfying (\ref{spectral2}) (which in particular satisfies that its tangential gradient vanishes on $\partial\Omega$). Then $(0,\overline p,\overline\omega)=(0,0,e^{-\mu t}\nabla z)$ fulfills (\ref{freesystem}) with $(\overline y_0,\overline\omega_0)=(0,\nabla z)$.
\end{remark}


It will be said that system \eqref{system} is locally exactly controllable to the trajectory $(0, \overline{p} , \overline{\omega})$ at time $T$ if there exists $\delta>0$ such that, for any initial data $(y_0,\omega_0)\in V\times H^1_0(\Omega)$ satisfying
$$
\left\|(y_0, \omega_0)-(0, \overline{\omega}_0)\right\|_{V\times H^1(\Omega)} \leq \delta,
$$
there exists a control $(u,v)$ and an associated solution $(y,\omega,p)$ such that
$$
y(T, \cdot)=0 \text{ and } \omega(T, \cdot)=\overline{\omega}(T) \text{ in } \Omega.
$$

We can now state the main results of this paper.
\begin{theorem}\label{Control3}
Assume that $d=3$ and $\overline\omega$ satisfies (\ref{overlineomega}). Then, (\ref{system}) is locally exactly controllable with control $(0,v)$ where $v\in L^2(Q)$.
\end{theorem}

\begin{theorem}\label{Control2}
Assume that $d=2$ and $\overline\omega\equiv 0$. Then, (\ref{system}) is locally exactly controllable with  control $(u,0)$ where $u\in L^2(Q)$.
\end{theorem}

The local exact controllability to any (sufficiently regular) trajectory $(\overline y,\overline p,\overline\omega)$ of (\ref{system}) has been obtained in \cite{FCG} whenever both controls $u$ and $v$ are active.

%
%

\vskip0.2cm Our main strategy relies on the null controllability of a linearized system around $(0, \overline{p}, \overline{\omega})$. It is classical that this null controllability result is equivalent to the  observability of the adjoint system. We will consequently consider the following problem:
\begin{equation}\label{adjoint}
\left\{
\begin{array}{llll}
     -\varphi_t-\Delta \varphi +\nabla \pi&=&P_1\psi+(d-2)(\nabla \psi)^T  \overline{\omega}+g_0 \qquad  &\text{ in } Q,  \\
    -\psi_t-\Delta \psi - (d-2)\nabla(\nabla \cdot \psi)&=& \nabla \times \varphi +g_1 \qquad &\text{ in } Q,\\
 \nabla \cdot \varphi &=& 0\qquad &\text{ in } Q,\\
     \varphi &=& 0 \qquad  &\text{ on } \Sigma, \\
     \psi&=&0  \qquad &\text{ on } \Sigma,\\
     \varphi(T,\cdot)&=&\varphi_T &\text{ in } \Omega, \\
     \psi(T,\cdot)&=&\psi_T &\text{ in } \Omega,\\
\end{array}
\right.
\end{equation}
where $\varphi_T \in H$ and $\psi_T \in L^2(\Omega)$.

\begin{remark} Our result in dimension $d=3$ deals with the control of (\ref{system}) through the fluid velocity but one could also be interested in controlling with the sole control $u$.

However, in the particular case of $(\overline y,\overline p,\overline\omega)=(0,0,0)$ one can prove that the associated linearized problem is not null-controllable when $\Omega$ is a ball.
 In fact, this linearized system is not even approximately controllable since the unique continuation property for the solutions of (\ref{adjoint}) (with $g_0\equiv g_1\equiv 0$)
\begin{equation}\label{UC}
\varphi=0\hbox{ in }(0,T)\times {\cal O}\Rightarrow \varphi\equiv \psi\equiv 0 \hbox{ in }Q
\end{equation}
is not satisfied. Indeed, if $(\varphi,\pi,\psi):=(0,0,e^{\mu t}\nabla z)$ where $z$ is a radially symetric solution of (\ref{spectral2}), then $(\varphi,\pi,\psi)$ is a solution of (\ref{adjoint}) which does not satisfy (\ref{UC}).
\end{remark}

The rest of the article is structured as follows: in the first part, we develop a strategy to prove two Carleman estimates adapted to the linear adjoint systems.
In the second part, we prove the observability of the linear adjoint systems and deduce the local controllability of the semilinear systems.
\section{Carleman estimates}

\subsection{Statement of the Carleman Inequalities}

We first set some notations. Let $\Omega_0$ be an open set satisfying $\overline{\Omega}_0\subset {\cal O}$ and $\eta\in {\cal C}^2\left(\overline{\Omega}\right) $ be a function such that
$$ \eta >0 \text{ in } \Omega, \ \eta=0 \text{ on } \partial \Omega, |\nabla \eta| >0 \text{ in }  \overline{\Omega} \backslash{\Omega_0}.$$
The existence of such a function $\eta$ is proved in \cite{FI} (see also \cite[Lemma 2.1]{IPY}). As usual in the context of Carleman estimates, we also define the following weight functions
$$
\begin{array}{c}\displaystyle
\alpha(t,x):= \frac{e^{2 \lambda \|\eta \|_{L^{\infty}(\Omega)}}-e^{\lambda \eta(x)}}{\ell(t)^m}\\ \noalign{\medskip}\displaystyle
\xi(t,x):= \frac{e^{\lambda \eta(x)}}{\ell(t)^m}
\end{array}
$$
where $\lambda \geq 1$  is a large constant to be fixed later, $m$ is an integer and $\ell: [0,T] \to [0, \infty)$ is some ${\cal C}^\infty$ function (first introduced in \cite{FI}) such that
$\ell>0$ in $(0,T)$, $\ell$ is constant in $[3T/8,5T/8]$, reaches a maxima at $t=T/2$ and
\begin{equation}\label{ell}
\forall t \in \left[0, \frac{T}{4}\right], \ \ell(t) = t,  \ \forall t \in \left[\frac{3T}{4},T\right], \ \ell(t) = T-t.
\end{equation}

In the sequel, we define $\alpha^*$ as the supremum of $\alpha$ in $\Omega$ (which is also its value on $\partial \Omega$).


We shall now state the two main Carleman estimates of the paper:
\begin{proposition}\label{Carleman3}
Let $d=3$, $m=8$ and $\overline{\omega} \in  L^\infty \left(0,T; W^{1,3+\delta}(\Omega)\right)\cap H^1\left(0,T; L^3(\Omega)\right)$ for some $\delta>0$. Then, for any $T>0$, there exist $C>0$ and $s_0>0$ such that for every $s\geq s_0$, the following inequality is satisfied for every $g_0 \in L^2(0,T; V)$ and every $g_1 \in L^2(0,T;L^2(\Omega))$,
\begin{equation}\label{CI}
\begin{array}{l}\displaystyle
s^2 \int_Q e^{-2s\alpha} \xi^2 |\psi|^2+\int_Q e^{-2s\alpha} |\varphi|^2
\\ \noalign{\medskip}\displaystyle
\hskip2cm\leq C\left(s^{-3} \int_Q e^{-2s \alpha} \xi^{-3}\left( |g_0|^2+|\nabla g_0|^2\right)+ \int_Qe^{-2s\alpha} |g_1|^2 + s^4 \int_{Q_{\cal O}} e^{-2s\alpha}\xi^{4} |\psi|^2 \right).
\end{array}
\end{equation}
where $Q_{\cal O}= (0,T)\times {\cal O}$ and $(\varphi,\psi)$ is any solution of (\ref{adjoint}).
\end{proposition}

\begin{proposition}\label{Carleman2}
Let $d=2$ and $m \geq 6$. For any $T>0$, there exist $C>0$ and $s_0>0$ such that for every $s\geq s_0$, the following inequality is satisfied for every $g_0 \in L^2(0,T; H^2(\Omega)\cap V)$ and every $g_1 \in L^2(0,T;H^2(\Omega)\cap H^1_0(\Omega))$,
\begin{equation}\label{CI9}
\begin{array}{l}\displaystyle
s^{-1} \int_Q e^{-2s\alpha} \xi^{-1} |\Delta\varphi|^2+s^{-2} \int_Q e^{-2s\alpha} \xi^{-2} |\Delta\psi|^2
\\ \noalign{\medskip}\displaystyle
\hskip2cm\leq C\left(s^{-2} \int_Q e^{-2s \alpha} \xi^{-2}\left( |g_0|^2+|\nabla g_0|^2+|\nabla^2 g_0|^2\right)\right.
\\ \noalign{\medskip}\displaystyle
\hskip2cm
\left.
+s^{-2} \int_Q e^{-2s \alpha} \xi^{-2}\left( |g_1|^2+|\nabla g_1|^2+|\nabla^2 g_1|^2\right)
+ s^{15} \int_{Q_{\cal O}} e^{-2s\alpha}\xi^{15} |\varphi|^2 \right).
\end{array}
\end{equation}
where $(\varphi,\psi)$ is any solution of (\ref{adjoint}).
\end{proposition}

\subsection{Proof of Proposition \ref{Carleman3}}

Our proof will rely on the Carleman estimate developped in \cite{IPY} and  on classical regularity estimates for the heat and Stokes systems (see Lemmata \ref{estimationM} and \ref{estimationL}).

More precisely, in order to avoid the pressure we will be led to apply some differential operators (such as $\nabla\times$ or $\Delta$) to our system so the new variables will not have prescribed boundary values. We will estimate these new variables thanks to the results of \cite{IPY} and \cite{FGGP}, where Carleman inequalities adapted to this situation are established. Finally, the boundary terms appearing will be absorbed by the left-hand side terms using regularity estimates for our system.

Throughout the proof, we will use the anisotropic Sobolev space
$$
H^{1/2,1/4}(\Sigma):=L^2(0,T;H^{1/2}(\partial\Omega))\cap H^{1/4}(0,T;L^2(\partial\Omega)).
$$
From its definition and standard trace estimates (see \cite{LionsMagenes}), one gets that if $f \in L^2(0,T;H^1(\Omega))$ is such that $\partial_t f \in L^2(0,T;H^{-1}(\Omega))$ then $f \in H^{1/2,1/4}(\Sigma)$ and
\begin{equation}\label{SobolevAnisotrope}
\|f\|_{H^{1/2,1/4}(\Sigma)} \lesssim \|f\|_{L^2(0,T;H^1(\Omega))}+\|\partial_t f\|_{L^2(0,T;H^{-1}(\Omega))} .
\end{equation}
Here and in the sequel, we use the notation $a \lesssim b$ to indicate the existence of a constant $C>0$ depending only on $\Omega$, ${\cal O}$ and $T$ such that $a \leq C b$.



\subsubsection{Estimate of $\psi $}

We first apply the divergence operator to the second equation of (\ref{adjoint}), which gives (since this operator commutes with the usual Laplacian operator):
$$
-\partial_t (\nabla \cdot  \psi)-2 \Delta(\nabla \cdot  \psi)=\nabla \cdot  g_1.
$$
For this nonhomogeneous heat equation, we apply the Carleman estimate presented in \cite[Theorem 2.1]{IPY}:
\begin{eqnarray}\label{equation1}s^{-1}\int_Q e^{-2 s\alpha} \xi^{-1} |\nabla (\nabla \cdot  \psi) |^2 &\lesssim& s^{-1/2} \left\|e^{-s\alpha}\xi^{-1/4} \nabla \cdot  \psi \right\|_{H^{1/2,1/4}(\Sigma)}^2+s^{-1/2}\left\|e^{-s\alpha}\xi^{-1/8} \nabla \cdot  \psi \right\|_{L^2(\Sigma)}^2\nonumber \\
&+&  \int_Q e^{-2s\alpha} |g_1|^2 + s\int_{Q_1}e^{-2s\alpha} \xi |\nabla \cdot  \psi|^2
\end{eqnarray}
for $s\gtrsim 1$, where $Q_1:=(0,T)\times \Omega_1$ and $\Omega_1$ is any non empty open subset such that $\overline{\Omega}_0 \subset \Omega_1$ and $\overline{\Omega}_1\subset{\cal O}$. 	

Moreover, since $\psi$ satisfies the system
\begin{equation*}
 \left \{
\begin{array}{ccc}
   (-\partial_t-\Delta) \psi&=&\nabla(\nabla \cdot \psi)+\nabla \times \varphi+g_1\\
   \psi&=&0
  \end{array} \right.
  \begin{array}{c}
   \text{in } Q,\\
   \text{on } \Sigma,
  \end{array}
\end{equation*}
a classical Carleman estimate for the heat equation (see e.g. \cite{FI}) gives us:
\begin{eqnarray}\label{equation2}s^2 \int_Q e^{-2s \alpha} \xi^2 |\psi|^2+\int_Q e^{-2s\alpha} |\nabla \psi|^2+s^{-2}\int_Q e^{-2s\alpha}\xi^{-2} |\nabla \nabla \psi|^2 &\lesssim& s^{-1} \int_Q e^{-2s \alpha} \xi^{-1} |\nabla(\nabla \cdot  \psi)|^2\nonumber\\
&+& s^{-1} \int_Q e^{-2s\alpha} \xi^{-1} \left( |\nabla \times \varphi|^2+|g_1|^2\right) \nonumber \\
&+&s^2 \int_{Q_1} e^{-2s\alpha} \xi^2 |\psi|^2
\end{eqnarray}
for any $s\gtrsim 1$. Consequently, a combination of \eqref{equation1} and \eqref{equation2} yields the estimate
\begin{equation}\label{equation3}
\begin{array}{l}\displaystyle
s^2 \int_Q e^{-2s \alpha} \xi^2 |\psi|^2+\int_Q e^{-2s\alpha} |\nabla \psi|^2+s^{-2}\int_Q e^{-2s\alpha}\xi^{-2} |\nabla \nabla \psi|^2
+s^{-1}\int_Q e^{-2 s\alpha} \xi^{-1} |\nabla (\nabla \cdot  \psi) |^2
\\ \noalign{\medskip}\displaystyle
\lesssim  s^{-1/2} \left\|e^{-s\alpha}\xi^{-1/4} \nabla \cdot  \psi \right\|_{H^{1/2,1/4}(\Sigma)}^2
+s^{-1/2}\left\|e^{-s\alpha}\xi^{-1/8} \nabla \cdot  \psi \right\|_{L^2(\Sigma)}^2
+ \int_Q e^{-2s\alpha} |g_1|^2
\\ \noalign{\medskip}\displaystyle
+s^{-1} \int_Qe^{-2s\alpha} \xi^{-1} |\nabla \times \varphi|^2
+s \int_{Q_1}e^{-2s\alpha} \xi |\nabla \cdot \psi|^2
+s^2 \int_{Q_1} e^{-2s\alpha} \xi^2 |\psi|^2.
\end{array}
\end{equation}
Furthermore, if $Q_2$ is any open subset of $Q$ of the form $ (0,T)\times \Omega_2$ such that $\overline{\Omega}_1 \subset \Omega_2$ and $\overline{\Omega}_2 \subset {\cal O}$, an integration by parts easily gives
$$s\int_{Q_1} e^{-2s\alpha} \xi |\nabla \cdot  \psi|^2 \leq \varepsilon s^{-1} \int_{Q_2} e^{-2s\alpha} \xi^{-1} |\nabla(\nabla \cdot \psi)|^2+ C\varepsilon^{-1} s^3 \int_{Q_2} e^{-2s\alpha} \xi^3 |\psi|^2$$
for any $\varepsilon>0$ and some $C>0$.

\noindent Choosing $\varepsilon$ sufficiently small, one consequently gets from \eqref{equation3},
\begin{equation}\label{equation4}
\begin{array}{l}\displaystyle
s^2 \int_Q e^{-2s \alpha} \xi^2 |\psi|^2+\int_Q e^{-2s\alpha} |\nabla \psi|^2 +s^{-2}\int_Q e^{-2s\alpha}\xi^{-2} |\nabla \nabla \psi|^2
\\ \noalign{\medskip}\displaystyle
\lesssim B_1+ \int_Q e^{-2s\alpha} |g_1|^2+s^{-1} \int_Qe^{-2s\alpha} \xi^{-1} |\nabla \times \varphi|^2
+s^3 \int_{Q_2} e^{-2s\alpha} \xi^3 |\psi|^2
\end{array}
\end{equation}
where $B_1$ stands for the trace terms
$$s^{-1/2} \left\|e^{-s\alpha}\xi^{-1/4} \nabla \cdot  \psi \right\|_{H^{1/2,1/4}(\Sigma)}^2+s^{-1/2}\left\|e^{-s\alpha}\xi^{-1/8} \nabla \cdot  \psi \right\|_{L^2(\Sigma)}^2.$$
We shall now prove the following estimate:
\begin{equation}\label{traces}
B_1\leq  \varepsilon s^2 \int_Q e^{-2s\alpha} \xi^2 |\psi|^2+C \left(s^{-1/2}\int_Q e^{-2s\alpha} \xi^{-1/4} |g_1|^2+ s^{-1/2}\int_Qe^{-2s\alpha} \xi^{-1/4} |\nabla \times \varphi|^2\right)
\end{equation}
for any $\varepsilon>0$ and some $C>0$ (which may depend on $\varepsilon$).

To do so, let us consider $\xi^*=\xi_{|\Sigma}$ and  define the weight function $\sigma_0 (t):=s^{-1/4}(\xi^*(t))^{-1/4} e^{-s \alpha^*(t)}$.
Straightforward computations show that
\begin{equation*}
 \left \{
\begin{array}{ccc}
   -\partial_t(\sigma_0  \psi)-\Delta(\sigma_0 \psi) - \nabla(\nabla. (\sigma_0 \psi))&=&-\sigma_0 ' \psi +\sigma_0 \nabla \times \varphi+\sigma_0  g_1\\
   \sigma_0  \psi&=&0\\
   (\sigma_0  \psi)(T,\cdot)&=&0
  \end{array} \right.
  \begin{array}{c}
   \text{in } Q,\\
   \text{on } \Sigma,\\
   \text{in } \Omega
  \end{array}
\end{equation*}
and consequently, thanks to \eqref{SobolevAnisotrope},
$$s^{-1/2} \left\|e^{-s\alpha}\xi^{-1/4} \nabla \cdot  \psi \right\|_{H^{1/2,1/4}(\Sigma)}^2 \lesssim \| \sigma_0  \nabla\cdot \psi\|^2_{L^2(0,T;H^1(\Omega))}+\| \sigma_0  \psi\|^2_{H^1(0,T;L^{2}(\Omega))}.$$
Since $|\sigma_0 '| \lesssim s^{3/4} (\xi^*)^{7/8} e^{-s \alpha^*}$, one deduces using Lemma \ref{estimationM} (a) that for $s\gtrsim 1$,
 \begin{eqnarray*} \|\sigma_0  \psi \|^2_{L^2(0,T;H^2(\Omega))}+\|\sigma_0  \psi \|^2_{H^1(0,T;L^2(\Omega))} &\lesssim& \int_Q \left((\sigma_0 ')^2 |\psi|^2+ \sigma_0 ^2 |\nabla \times \varphi|^2+ \sigma_0 ^2 |g_1|^2\right)\\
  &\leq & C \left( s^{-1/2} \int_Q e^{-2s\alpha} \xi^{-1/2}|g_1|^2+ s^{-1/2}\int_Qe^{-2s\alpha}  \xi^{-1/2}|\nabla \times \varphi|^2\right)\\
&+& \varepsilon  s^{2} \int_Q e^{-2s\alpha} \xi^{2} |\psi|^2.
  \end{eqnarray*}
The estimate of the second term of $B_1$ is simpler, so we omit its proof. This concludes the proof of \eqref{traces}.

Finally, one immediately deduces from the last computations and \eqref{equation4} the estimate
\begin{eqnarray}\label{bilan1}
I(\psi) \lesssim \int_Q e^{-2s\alpha} |g_1|^2+ s^{-1/2} \int_Qe^{-2s\alpha} \xi^{-1/2} |\nabla \times \varphi|^2 +s^3 \int_{Q_2} e^{-2s\alpha} \xi^3 |\psi|^2 .
\end{eqnarray}
where
$$
I(\psi):= s^{-1/2} \int_Q e^{-2s \alpha^*} (\xi^*)^{-1/2} |\psi_t|^2+s^2 \int_Q e^{-2s \alpha} \xi^2 |\psi|^2+\int_Q e^{-2s\alpha} |\nabla \psi|^2+s^{-2}\int_Q e^{-2s\alpha}\xi^{-2} |\nabla \nabla \psi|^2 .
$$
In order to get an estimate in terms of a local term of $\psi$ only, our next goal is to get rid of the term $\displaystyle s^{-1/2}\int_Qe^{-2s\alpha}\xi^{-1/2} |\nabla \times \varphi|^2$.

\subsubsection{Estimate of the global term in $\nabla \times \varphi$}
To do so, we first apply the curl operator then the gradient operator to the first equation of (\ref{adjoint}). One easily gets
$$
-\partial_t\left(\nabla(\nabla \times \varphi)\right)-\Delta\left(\nabla(\nabla \times \varphi)\right)=\nabla\left(\nabla \times \nabla \times \psi+\nabla \times \left[(\nabla \psi)^T \overline{\omega}\right]+\nabla \times g_0\right).
$$
 We apply again \cite[Theorem 2.1]{IPY} with different powers of $\xi$. More precisely, we apply that Carleman estimate to $s^{-3/2}\xi^{-3/2}\nabla(\nabla\times\varphi)$ and we get
\begin{eqnarray} \label{equation5}
s^{-2} \int_Q e^{-2s\alpha} \xi^{-2} |\nabla(\nabla \times \varphi)|^2&+&s^{-4} \int_Q e^{-2s\alpha} \xi^{-4} |\nabla \nabla(\nabla \times \varphi)|^2 \nonumber\\
&\lesssim& s^{-7/2} \left\|e^{-s\alpha} \xi^{-7/4} \nabla(\nabla \times \varphi)\right\|_{H^{1/2,1/4}(\Sigma)}^2+s^{-7/2}\left\|e^{-s\alpha} \xi^{-13/8} \nabla(\nabla \times \varphi)\right\|_{L^2(\Sigma)}^2\nonumber\\
&+& s^{-3} \int_Q e^{-2s\alpha}\xi^{-3} \left(s^2 \xi^2|\nabla \psi|^2+ |\nabla \nabla \psi|^2\right)+s^{-3} \int_Q e^{-2s\alpha} \xi^{-3} |\nabla \times g_0|^2\nonumber\\
&+& s^{-2} \int_{Q_1}e^{-2s\alpha}\xi^{-2} |\nabla(\nabla \times \varphi)|^2.
\end{eqnarray}
Here, we have used (\ref{pot}) and the fact that $\overline{\omega} \in L^\infty(0,T; W^{1,3}(\Omega))$.

Using Lemma \ref{dominationH1} for $u:=\nabla\times\varphi$, one directly deduces from \eqref{equation5},
\begin{eqnarray}\label{equation6}
J(\varphi) &\lesssim& s^{-7/2} \left\|e^{-s\alpha} \xi^{-7/4} \nabla(\nabla \times \varphi)\right\|_{H^{1/2,1/4}(\Sigma)}^2+s^{-7/2}\left\|e^{-s\alpha} \xi^{-13/8} \nabla(\nabla \times \varphi)\right\|_{L^2(\Sigma)}^2\nonumber\\
&+& s^{-3} \int_Q e^{-2s\alpha}\xi^{-3}\left(s^2 \xi^2|\nabla \psi|^2+ |\nabla \nabla \psi|^2\right)+s^{-3} \int_Q e^{-2s\alpha} \xi^{-3} |\nabla \times g_0|^2\nonumber\\
&+& s^{-2} \int_{Q_1}e^{-2s\alpha}\xi^{-2} |\nabla(\nabla \times \varphi)^2|+ \int_{Q_1} e^{-2s\alpha} |\nabla \times \varphi|^2
\end{eqnarray}
where
$$J(\varphi):=\int_Qe^{-2s\alpha} |\nabla \times \varphi|^2+s^{-2} \int_Q e^{-2s\alpha} \xi^{-2} |\nabla(\nabla \times \varphi)|^2 +s^{-4} \int_Q e^{-2s\alpha} \xi^{-4} |\nabla \nabla(\nabla \times \varphi)|^2.$$
Moreover, an integration by parts gives us in the same way as above
$$
s^{-2}\int_{Q_1} e^{-2s\alpha} \xi^{-2} |\nabla(\nabla \times \varphi)|^2 \leq \varepsilon s^{-4} \int_{Q_2} e^{-2s\alpha} \xi^{-4} |\nabla \nabla(\nabla \times \varphi)|^2+ C\int_{Q_2} e^{-2s\alpha} |\nabla \times \varphi|^2
$$
for any $\varepsilon>0$ and some $C>0$ depending on $\varepsilon$.  This allows us, by an appropriate choice of $\varepsilon>0$, to get from \eqref{equation6}
\begin{eqnarray}\label{equation7}
J(\varphi) \lesssim B_2+  s^{-3} \int_Q e^{-2s\alpha}\xi^{-3} \left(s^2\xi^2|\nabla \psi|^2+ |\nabla \nabla \psi|^2\right)+s^{-3} \int_Q e^{-2s\alpha} \xi^{-3} |\nabla \times g_0|^2+ \int_{Q_2} e^{-2s\alpha} |\nabla \times \varphi|^2
\end{eqnarray}
where $B_2$ stands for the trace terms
$$s^{-7/2} \left\|e^{-s\alpha} \xi^{-7/4} \nabla(\nabla \times \varphi)\right\|_{H^{1/2,1/4}(\Sigma)}^2+s^{-7/2}\left\|e^{-s\alpha} \xi^{-13/8} \nabla(\nabla \times \varphi)\right\|_{L^2(\Sigma)}^2 .$$
We shall now prove the following estimate:
\begin{equation}\label{traces2}
B_2 \leq \varepsilon \left(\int_Q e^{-2s\alpha} |\nabla \times \varphi|^2+I(\psi)\right)+C s^{-7/2} \int_Q e^{-2s\alpha} \xi^{-13/4} \left(|\nabla g_0|^2+|g_0|^2\right)
\end{equation}
for any $\varepsilon>0$ and some $C>0$ (which may depend on $\varepsilon$).

To do so, we define the weight function $\sigma_1 (t):=s^{-7/4}(\xi^*(t))^{-13/8} e^{-s \alpha^*(t)}$ and consider the system satisfied by $\sigma_1  \varphi$:
\begin{equation*}
 \left \{
\begin{array}{ccc}
   (-\partial_t-\Delta)(\sigma_1  \varphi)+\sigma_1  \nabla \pi &=&-\sigma_1 ' \varphi +\sigma_1  g_0 +\sigma_1 \left(\nabla \times \psi+(\nabla \psi)^T  \overline{\omega}\right)\\
   \nabla\cdot(\sigma_1  \varphi)&=&0\\
   \sigma_1  \varphi&=&0\\
   (\sigma_1 \psi)(T,\cdot)&=&0
  \end{array} \right.
  \begin{array}{c}
   \text{in } Q,\\
    \text{in } Q,\\
   \text{on } \Sigma,\\
   \text{in } \Omega.
  \end{array}
\end{equation*}
Thanks to \eqref{SobolevAnisotrope}, we first get
$$s^{-7/2} \left\|e^{-s\alpha}\xi^{-13/8} \nabla( \nabla \times \varphi) \right\|_{H^{1/2,1/4}(\Sigma)}^2 \lesssim \| \sigma_1  \varphi\|^2_{L^2(0,T;H^3(\Omega))}+\| \sigma_1  \psi\|^2_{H^1(0,T;H^1(\Omega))}.$$
Then, we apply Lemma \ref{estimationL} (b) with $h_V:=-\sigma_1' \varphi+\sigma_1 g_0$ and $h:=\sigma_1 \left(\nabla \times \psi+(\nabla \psi)^T  \overline{\omega}\right)$. One obtains:
$$
\begin{array}{l}\displaystyle
\| \sigma_1  \varphi\|^2_{L^2(0,T;H^3(\Omega))}+\| \sigma_1  \psi\|^2_{H^1(0,T;H^1(\Omega))}
\\ \noalign{\medskip}\displaystyle
\hskip2cm\lesssim \| \sigma_1'  \varphi\|^2_{L^2(0,T;V)}+\| \sigma_1  g_0\|^2_{L^2(0,T;V)}+\| \sigma_1 \nabla \nabla \psi\|^2_{L^2(0,T;L^2(\Omega))}+\| \sigma_1  \psi\|^2_{H^1(0,T;L^2(\Omega))},
\end{array}
$$
where we have used Lemma \ref{estimationL2H-1} in order to estimate the second term in the definition of $h$.


\noindent For $s$ large enough, we consequently find
\begin{eqnarray*}
s^{-7/2} \left\|e^{-s\alpha}\xi^{-13/8} \nabla( \nabla \times \varphi) \right\|_{H^{1/2,1/4}(\Sigma)}^2 &\leq& \varepsilon \left( \int_Q e^{-2s\alpha} |\nabla \times \varphi|^2+ I(\psi)\right)\nonumber \\
&+& C s^{-7/2} \int_Q e^{-2s\alpha} \xi^{-13/4} \left( |g_0|^2+ |\nabla g_0 |^2\right).
\end{eqnarray*}
This concludes the proof of \eqref{traces2}.

Using \eqref{traces2}, we now infer from \eqref{equation7} that, for any $\varepsilon>0$ there exists $C>0$ such that
\begin{eqnarray}\label{bilan2}
J(\varphi) &\leq& \varepsilon I(\psi)+C \int_{Q_2} e^{-2s\alpha} |\nabla \times \varphi|^2 \nonumber \\
&+&  C \left( s^{-3} \int_Q e^{-2s\alpha}\xi^{-3}\left(s^2 \xi^2|\nabla \psi|^2+ |\nabla \nabla \psi|^2\right) +s^{-3} \int_Q e^{-2s\alpha}\xi^{-3} \left(|g_0|^2+|\nabla g_0|^2\right)\right).
\end{eqnarray}
Combining \eqref{bilan1} and \eqref{bilan2} and choosing an appropriate value of $\varepsilon>0$, one can now conclude that
\begin{eqnarray}\label{bilan3}
I(\psi)+J(\varphi)&\lesssim& s^{-3}\int_Q e^{-2s\alpha} \xi^{-3} \left(|g_0|^2+|\nabla \times g_0|^2\right)+ \int_Q e^{-2s\alpha} |g_1|^2\nonumber\\
&+& \int_{Q_2} e^{-2s\alpha} |\nabla \times \varphi|^2+s^3 \int_{Q_2} e^{-2s\alpha} \xi^3 |\psi|^2
\end{eqnarray}
for $s$ large enough.

\subsubsection{Estimate of the local term in $\nabla \times \varphi$}
Using the second equation of (\ref{adjoint}), one first has
$$
\int_{Q_2} e^{-2s\alpha} |\nabla \times \varphi|^2 \leq\int_{Q} \eta_2 e^{-2s\alpha} |\nabla \times \varphi|^2  =\int_{Q} \eta_2 e^{-2s\alpha} (\nabla \times \varphi)\cdot(-\psi_t-\Delta \psi -\nabla(\nabla\cdot \psi)-g_1),
$$
where $\eta_2: \Omega \to [0,+\infty[$ is some non--negative regular function supported in ${\cal O}$ such that $\eta_2=1$ on $\Omega_2$.
Similarly as above, integrations by parts now show that
\begin{eqnarray*}
\int_{Q_2} e^{-2s\alpha} |\nabla \times \varphi|^2 &\leq& \int_{Q} \eta_2 e^{-2s \alpha} (\nabla \times \varphi_t)\cdot \psi + C\int_{Q_{\cal O}}e^{-2s\alpha} |\nabla \times \varphi||g_1|\\
&+&C\left(  \int_{Q_{\cal O}} e^{-2s \alpha} |\psi| \left(|\nabla (\nabla \times \varphi)|+|\nabla \nabla (\nabla \times \varphi)|\right)+ s^2 \int_{Q_{\cal O}} e^{-2s\alpha} \xi^{2} |\nabla \times \varphi| |\psi|\right)\\
&\leq&  \int_{Q} \eta_2 e^{-2s \alpha} (\nabla \times \varphi_t)\cdot \psi + \varepsilon J(\varphi)+ C \left( s^4 \int_Qe^{-2s\alpha } \xi^4 |\psi|^2 + \int_Q e^{-2s\alpha} |g_1|^2\right)
\end{eqnarray*}
for some $C>0$ which may depend on $\varepsilon$. Moreover, applying the curl operator to the first equation of $(S')$ and using (\ref{pot}) and Young's inequality , one has
\begin{eqnarray*}
\int_{Q} \eta_2 e^{-2s \alpha} (\nabla \times \varphi_t)\cdot\psi &=&- \int_Q \eta_2 e^{-2s \alpha} \left( \Delta (\nabla \times \varphi)+ \nabla \times \nabla \times \psi+\nabla \times [(\nabla \psi)^T \overline{\omega}]+\nabla \times g_0\right)\cdot \psi\\
& \leq & \varepsilon \left( s^{-4} \int_Q e^{-2s \alpha} \xi^{-4} |\nabla \nabla (\nabla \times \varphi)|^2+s^{-2} \int_Q e^{-2s \alpha} \xi^{-2} \left(|\nabla \psi|^2+ |\nabla \nabla \psi|^2\right)\right)\\
&+& C \left(s^{-4} \int_{Q} e^{-2s \alpha} \xi^{-4} |\nabla g_0|^2+ s^{4} \int_{Q_{\cal O}}  e^{-2s \alpha} \xi^{4} |\psi|^2\right)
\end{eqnarray*}
for some $C>0$ which may depend on $\varepsilon$.

Finally, we have proved that for any $\varepsilon>0$ there exists $C>0$ such that
$$
\int_{Q_2} e^{-2s\alpha} |\nabla \times \varphi|^2 \leq \varepsilon\left(I(\psi)+J(\varphi)\right)+ C\left( s^{-4} \int_{Q} e^{-2s \alpha} \xi^{-4} |\nabla g_0|^2+\int_Q e^{-2s\alpha} |g_1|^2+s^{4} \int_{Q_{\cal O}}  e^{-2s \alpha} \xi^{4} |\psi|^2\right).
$$
Using now \eqref{bilan3}, the proof of Proposition \ref{Carleman3} is complete.

\subsection{Proof of Proposition~\ref{Carleman2}}

We apply the Laplacian operator to the first equation of (\ref{adjoint}). Since
$$
\Delta\pi = \nabla\cdot g_0,
$$
we get :
$$
-(\Delta\varphi)_t-\Delta(\Delta\varphi) = P_1\Delta\psi + \Delta g_0 - \nabla(\nabla\cdot g_0)\hbox{ in }Q.
$$
We now apply \cite[Theorem 1]{FGGP} to $\Delta\varphi$ and obtain
\begin{equation}\label{Carlemanvarphi}
\begin{array}{l}\displaystyle
s^{-3}\int_Q e^{-2s\alpha}\xi^{-3}|\nabla(\Delta\varphi)|^2+s^{-1}\int_Qe^{-2s\alpha}\xi^{-1}|\Delta\varphi|^2
\\ \noalign{\medskip}\displaystyle
\hskip3cm\lesssim s^{-3}\int_{\Sigma}e^{-2s\alpha}\xi^{-3}\left|\frac{\partial\Delta\varphi}{\partial\nu}\right|^2
+s^{-2}\int_Qe^{-2s\alpha}\xi^{-2}(|g_0|^2+|\nabla g_0|^2)
\\ \noalign{\medskip}\displaystyle
\hskip3cm+s^{-1}\int_{Q_1}e^{-2s\alpha}\xi^{-1}|\Delta\varphi|^2
+s^{-2}\int_{Q}e^{-2s\alpha}\xi^{-2}|\Delta\psi|^2
\end{array}
\end{equation}
for $s$ large enough.

We now apply the Laplacian operator to the second equation of (\ref{adjoint}). We get :
$$
-(\Delta\psi)_t-\Delta(\Delta\psi) = \nabla\times\Delta\varphi + \Delta g_1\hbox{ in }Q.
$$
We then apply \cite[Theorem 1]{FGGP} to $\Delta\psi$, which gives :
\begin{equation}\label{Carlemanpsi}
\begin{array}{l}\displaystyle
s^{-4}\int_Q e^{-2s\alpha}\xi^{-4}|\nabla(\Delta\psi)|^2+s^{-2}\int_Qe^{-2s\alpha}\xi^{-2}|\Delta\psi|^2
\\ \noalign{\medskip}\displaystyle
\hskip3cm\lesssim s^{-4}\int_{\Sigma}e^{-2s\alpha}\xi^{-4}\left|\frac{\partial}{\partial\nu}\Delta\psi\right|^2
+s^{-3}\int_Qe^{-2s\alpha}\xi^{-3}(|g_1|^2+|\nabla g_1|^2)
\\ \noalign{\medskip}\displaystyle
\hskip3cm+s^{-2}\int_{Q_1}e^{-2s\alpha}\xi^{-2}|\Delta\psi|^2
+s^{-3}\int_{Q}e^{-2s\alpha}\xi^{-3}|\Delta\varphi|^2.
\end{array}
\end{equation}
Combining (\ref{Carlemanvarphi}) and (\ref{Carlemanpsi}), we find
\begin{equation}\label{Carlemanvarpsi}
\begin{array}{l}\displaystyle
s^{-3}\int_Q e^{-2s\alpha}\xi^{-3}|\nabla(\Delta\varphi)|^2+s^{-1}\int_Qe^{-2s\alpha}\xi^{-1}|\Delta\varphi|^2
+s^{-4}\int_Q e^{-2s\alpha}\xi^{-4}|\nabla(\Delta\psi)|^2+s^{-2}\int_Qe^{-2s\alpha}\xi^{-2}|\Delta\psi|^2
\\ \noalign{\medskip}\displaystyle
\hskip3cm\lesssim s^{-1}\int_{Q_1}e^{-2s\alpha}\xi^{-1}|\Delta\varphi|^2
+s^{-2}\int_{Q_1}e^{-2s\alpha}\xi^{-2}|\Delta\psi|^2+B_3+B_4
\\ \noalign{\medskip}\displaystyle
\hskip3cm+s^{-2}\int_Qe^{-2s\alpha}\xi^{-2}(|g_0|^2+|\nabla g_0|^2)
+s^{-3}\int_Qe^{-2s\alpha}\xi^{-3}(|g_1|^2+|\nabla g_1|^2),
\end{array}
\end{equation}
where
$$
B_3 : = s^{-3}\int_{\Sigma}e^{-2s\alpha}\xi^{-3}\left|\frac{\partial}{\partial\nu}\Delta\varphi\right|^2
$$
and
$$
B_4 : = s^{-4}\int_{\Sigma}e^{-2s\alpha}\xi^{-4}\left|\frac{\partial}{\partial\nu}\Delta\psi\right|^2.
$$

Before estimating the local and boundary terms, let us apply the classical Carleman estimate for the Laplace operator with homogeneous
Dirichlet boundary conditions :
$$
s^{-2}\int_Qe^{-2s\alpha}\xi^{-2}|\nabla\nabla\varphi|^2+s^{2}\int_Qe^{-2s\alpha}\xi^{2}|\varphi|^2\lesssim s^{2}\int_{Q_1}e^{-2s\alpha}\xi^{2}|\varphi|^2+s^{-1}\int_Qe^{-2s\alpha}\xi^{-1}|\Delta\varphi|^2.
$$
Combining with (\ref{Carlemanvarpsi}), we deduce :
\begin{equation}\label{Carlemanvarpsi2}
\begin{array}{l}\displaystyle
s^{-3}\int_Q e^{-2s\alpha}\xi^{-3}|\nabla(\Delta\varphi)|^2
+s^{-2}\int_Qe^{-2s\alpha}\xi^{-2}|\nabla\nabla\varphi|^2
+s^{-4}\int_Q e^{-2s\alpha}\xi^{-4}|\nabla(\Delta\psi)|^2+s^{-2}\int_Qe^{-2s\alpha}\xi^{-2}|\Delta\psi|^2
\\ \noalign{\medskip}\displaystyle
+s^{-1}\int_Qe^{-2s\alpha} \xi^{-1}|\Delta\varphi|^2\lesssim s^{2}\int_{Q_1}e^{-2s\alpha}\xi^{2}|\varphi|^2+s^{-1}\int_{Q_1}e^{-2s\alpha}\xi^{-1}|\Delta\varphi|^2
+s^{-2}\int_{Q_1}e^{-2s\alpha}\xi^{-2}|\Delta\psi|^2
\\ \noalign{\medskip}\displaystyle
+B_3+B_4+s^{-2}\int_Qe^{-2s\alpha}\xi^{-2}(|g_0|^2+|\nabla g_0|^2)
+s^{-3}\int_Qe^{-2s\alpha}\xi^{-3}(|g_1|^2+|\nabla g_1|^2).
\end{array}
\end{equation}


Let us now estimate the local and boundary terms in the right-hand side of (\ref{Carlemanvarpsi2}).

\subsubsection{Estimate of the local terms in (\ref{Carlemanvarpsi2})}

In this paragraph, we estimate the second and third terms in the right-hand side of (\ref{Carlemanvarpsi2}).

Regarding the term in $\psi$, we apply the curl operator to the equation satisfied by $\varphi$. This gives :
\begin{equation}\label{betis}
\Delta\psi = -(\nabla\times\varphi)_t-\Delta(\nabla\times\varphi)-\nabla\times g_0\hbox{ in }Q_1.
\end{equation}
Let  $\eta_1$ be a positive function satisfying
$$
\eta_1\in C^2_c(\Omega_2),\,\eta_1(x)=1\,\,\,\forall x\in \Omega_1.
$$
Using (\ref{betis}), we get the following splitting :
$$
s^{-2}\int_{Q_1}e^{-2s\alpha}\xi^{-2}|\Delta\psi|^2\leq s^{-2}\int_{Q}\eta_1e^{-2s\alpha}\xi^{-2}(\Delta\psi)(-(\nabla\times\varphi)_t-\Delta(\nabla\times\varphi)-\nabla\times g_0):=I_1+I_2+I_3.
$$

\paragraph{Estimate of $I_1$.} We integrate by parts with respect to $t$ to get
$$
I_1 = s^{-2}\int_Q\eta_1(e^{-2s\alpha}\xi^{-2})_t\Delta\psi\,\nabla\times\varphi + s^{-2}\int_Q\eta_1\,e^{-2s\alpha}\xi^{-2}\Delta\psi_t\,\nabla\times\varphi :=I_{1,1}+I_{1,2}.
$$
For the first term, using Young's inequality, we get
$$
|I_{1,1}|\lesssim s^{-1}\int_{Q_2}e^{-2s\alpha}\xi^{-1+1/m}|\Delta\psi||\nabla\times\varphi|\leq \varepsilon s^{-2}\int_Qe^{-2s\alpha}\xi^{-2}|\Delta\psi|^2+C\int_{Q_2}e^{-2s\alpha}\xi^{2/m}|\nabla\times\varphi|^2
$$
For the second term, we integrate by parts in $x$ and we obtain :
$$
\begin{array}{l}\displaystyle
|I_{1,2}|
\lesssim s^{-2}\int_{Q_2}|\nabla(e^{-2s\alpha}\xi^{-2})||\nabla\psi_t||\nabla\times\varphi|
+s^{-2}\int_{Q_2}e^{-2s\alpha}\xi^{-2}|\nabla\psi_t||\nabla\times\varphi|
\\ \noalign{\medskip}\displaystyle
\phantom{|I_{1,2}|}+s^{-2}\int_{Q_2}e^{-2s\alpha}\xi^{-2}|\nabla\psi_t||\nabla(\nabla\times\psi)|
\\ \noalign{\medskip}\displaystyle
\phantom{|I_{1,2}|}\lesssim s^{-1}\int_{Q_2}e^{-2s\alpha}\xi^{-1}|\nabla\psi_t||\nabla\times\varphi|
+s^{-2}\int_{Q_2}e^{-2s\alpha}\xi^{-2}|\nabla\psi_t||\nabla(\nabla\times\varphi)|
\\ \noalign{\medskip}\displaystyle
\phantom{|I_{1,2}|}\leq \varepsilon s^{-4}\int_{Q}e^{-2s\alpha}\xi^{-4}|\nabla\psi_t|^2
+Cs^{2}\int_{Q_2}e^{-2s\alpha}\xi^{2}(|\nabla(\nabla\times\varphi)|^2+|\nabla\times\varphi|^2).
\end{array}
$$
Consequently, this first term is estimated as follows :
\begin{equation}\label{Seville}
\begin{array}{l}\displaystyle
|I_1| \leq \varepsilon \left(s^{-2}\int_Qe^{-2s\alpha}\xi^{-2}|\Delta\psi|^2+s^{-4}\int_{Q}e^{-2s\alpha}\xi^{-4}|\nabla\psi_t|^2\right)
\\ \noalign{\medskip}\displaystyle
\phantom{|I_1|}
+Cs^2\int_{Q_2}e^{-2s\alpha}\xi^2(|\nabla\times\varphi|^2+|\nabla(\nabla\times\varphi)|^2).
\end{array}
\end{equation}

\paragraph{Estimate of $I_2$.} We integrate by parts with respect to $x$. We obtain
\begin{equation}\label{Benfica}
\begin{array}{l}\displaystyle
|I_2|\lesssim s^{-1}\int_Qe^{-2s\alpha}\xi^{-1}|\Delta\psi||\Delta\varphi|+s^{-2}\int_Q e^{-2s\alpha}\xi^{-2}|\nabla\times\Delta\psi||\Delta\varphi|
\\ \noalign{\medskip}\displaystyle
\phantom{|I_2|}
\leq \varepsilon\left(s^{-2}\int_Qe^{-2s\alpha}\xi^{-2}|\Delta\psi|^2+s^{-4}\int_{Q}e^{-2s\alpha}\xi^{-4}|\nabla\Delta\psi|^2\right)
+C\int_Qe^{-2s\alpha}|\Delta\varphi|^2.
\end{array}
\end{equation}

\paragraph{Estimate of $I_3$.} Using Young's inequality, we get
$$
\begin{array}{l}\displaystyle
|I_3|
\leq \varepsilon s^{-2}\int_Qe^{-2s\alpha}\xi^{-2}|\Delta\psi|^2+Cs^{-2}\int_{Q}e^{-2s\alpha}\xi^{-2}|\nabla\times g_0|^2.
\end{array}
$$
Putting this last inequality together with (\ref{Seville})-(\ref{Benfica}), we obtain
\begin{equation}\label{Atletico}
\begin{array}{r}\displaystyle
s^{-2}\int_{Q_1}e^{-2s\alpha}\xi^{-2}|\Delta\psi|^2\leq \varepsilon \left(s^{-2}\int_Qe^{-2s\alpha}\xi^{-2}|\Delta\psi|^2+s^{-4}\int_{Q}e^{-2s\alpha}\xi^{-4}(|\nabla\psi_t|^2+|\nabla\Delta\psi|^2)\right)
\\ \noalign{\medskip}\displaystyle
\hskip3cm+C\left(s^2\int_{Q_2}e^{-2s\alpha}\xi^2(|\nabla\times\varphi|^2+|\nabla(\nabla\times\varphi)|^2)+s^{-2}\int_{Q}e^{-2s\alpha}\xi^{-2}|\nabla g_0|^2\right).
\end{array}
\end{equation}
Using now the relation
 $$
 \nabla\psi_t = -\nabla\Delta\psi-\nabla(\nabla\times\varphi)-\nabla g_1 \hbox{ in }Q,
 $$
 the term $\displaystyle  s^{-4}\int_{Q}e^{-2s\alpha}\xi^{-4}|\nabla\psi_t|^2
 $ is bounded by the left-hand side of (\ref{Carlemanvarpsi2}). This allows to deduce
\begin{equation}\label{Carlemanvarpsi3}
\begin{array}{l}\displaystyle
s^{-3}\int_Q e^{-2s\alpha}\xi^{-3}|\nabla(\Delta\varphi)|^2
+s^{-2}\int_Qe^{-2s\alpha}\xi^{-2}|\nabla\nabla\varphi|^2
+s^{-4}\int_Q e^{-2s\alpha}\xi^{-4}|\nabla(\Delta\psi)|^2
\\ \noalign{\medskip}\displaystyle
+s^{-2}\int_Qe^{-2s\alpha}\xi^{-2}|\Delta\psi|^2+s^{-1}\int_Qe^{-2s\alpha} \xi^{-1}|\Delta\varphi|^2
+s^{-4}\int_{Q}e^{-2s\alpha}\xi^{-4}|\nabla\psi_t|^2+s^2\int_Qe^{-2\alpha}\xi^2|\varphi|^2
\\ \noalign{\medskip}\displaystyle
\lesssim s^{2}\int_{Q_2}e^{-2s\alpha}\xi^{2}(|\nabla(\nabla\times\varphi)|^2+|\nabla\times\varphi|^2+|\varphi|^2)
+B_3+B_4
\\ \noalign{\medskip}\displaystyle
+s^{-2}\int_Qe^{-2s\alpha}\xi^{-2}(|g_0|^2+|\nabla g_0|^2)
+s^{-3}\int_Qe^{-2s\alpha}\xi^{-3}(|g_1|^2+|\nabla g_1|^2).
\end{array}
\end{equation}

\paragraph{Estimate of the local term in $\varphi$.}

Let $\Omega_3$ be an open set such that $\overline{\Omega}_3\subset {\cal O} $ and $\overline{\Omega}_2 \subset \Omega_3$. After several integrations by parts, we get
\begin{equation}\label{localvarphi}
\begin{array}{l}\displaystyle
s^2\int_{Q_2}e^{-2s\alpha}\xi^2|\nabla(\nabla\times\varphi)|^2\leq \varepsilon s^{-3}\int_Qe^{-2s\alpha} \xi^{-3}|\nabla\Delta\varphi|^2+Cs^7\int_{Q_3}e^{-2s\alpha}\xi^7|\nabla\varphi|^2
\\ \noalign{\medskip}\displaystyle
\phantom{s^2\int_{Q_2}e^{-2s\alpha}\xi^2|\nabla(\nabla\times\varphi)|^2}
\leq \varepsilon \left(s^{-3}\int_Qe^{-2s\alpha} \xi^{-3}|\nabla\Delta\varphi|^2
+s^{-1}\int_Qe^{-2s\alpha} \xi^{-1}|\Delta\varphi|^2\right)
\\ \noalign{\medskip}\displaystyle
\phantom{s^2\int_{Q_2}e^{-2s\alpha}\xi^2|\nabla(\nabla\times\varphi)|^2}
+Cs^{15}\int_{Q_{\cal O}}e^{-2s\alpha}\xi^{15}|\varphi|^2.
\end{array}
\end{equation}
Putting this together with (\ref{Carlemanvarpsi3}), we deduce :
\begin{equation}\label{Carlemanvarpsi4}
\begin{array}{l}\displaystyle
s^2\int_Qe^{-2\alpha}\xi^2|\varphi|^2+s^{-1}\int_Qe^{-2s\alpha}\xi^{-1}|\Delta\varphi|^2
+s^{-2}\int_Qe^{-2s\alpha}\xi^{-2}|\Delta\psi|^2
+s^{-4}\int_{Q}e^{-2s\alpha}\xi^{-4}|\nabla\psi_t|^2
\\ \noalign{\medskip}\displaystyle
\lesssim B_3+B_4+s^{15}\int_{Q_{\cal O}}e^{-2s\alpha}\xi^{15}|\varphi|^2+s^{-2}\int_Qe^{-2s\alpha}\xi^{-2}(|g_0|^2+|\nabla g_0|^2)
\\ \noalign{\medskip}\displaystyle

+s^{-3}\int_Qe^{-2s\alpha}\xi^{-3}(|g_1|^2+|\nabla g_1|^2).
\end{array}
\end{equation}

Let us now prove that the term
$$
J_0:=s^{-1}\int_Qe^{-2s\alpha^*}(\xi^*)^{-1}|P_1\psi|^2
$$
is estimated by the left-hand side of (\ref{Carlemanvarpsi4}). For this, we use the equation satisfied by $\varphi$:
$$
J_0=s^{-1}\int_Qe^{-2s\alpha^*}(\xi^*)^{-1}(P_1\psi)\cdot(-\varphi_t-\Delta\varphi+\nabla\pi-g_0):=J_1+J_2+J_3+J_4.
$$
Observe that $J_3=0$ since $\psi=0$ on $\Sigma$. For the first term, we integrate by parts in time :

\begin{align*}
J_1&=s^{-1}\int_Qe^{-2s\alpha^*}(\xi^*)^{-1}(P_1\psi_t)\cdot\varphi+s^{-1}\int_Q(e^{-2s\alpha^*}(\xi^*)^{-1})'(P_1\psi)\cdot\varphi
\\
&\leq \frac{1}{4}J_0+2\left(s^{-4}\int_{Q}e^{-2s\alpha}\xi^{-4}|\nabla\psi_t|^2
+s^2\int_Qe^{-2\alpha}\xi^2|\varphi|^2
\right),
\end{align*}
for $m\geq 2$ and $s\gtrsim 1$. For the second and fourth terms, we have
$$
\begin{array}{l}\displaystyle
J_2+J_4
\leq \frac{1}{4}J_0+2\left(s^{-1}\int_{Q}e^{-2s\alpha}\xi^{-1}|\Delta\varphi|^2+s^{-1}\int_{Q}e^{-2s\alpha}\xi^{-1}|g_0|^2\right).
\end{array}
$$

Consequently, coming back to (\ref{Carlemanvarpsi4}), we obtain
\begin{equation}\label{Carlemanvarpsi5}
\begin{array}{l}\displaystyle
s^2\int_Qe^{-2\alpha}\xi^2|\varphi|^2+s^{-1}\int_Qe^{-2s\alpha}\xi^{-1}|\Delta\varphi|^2
+s^{-2}\int_Qe^{-2s\alpha}\xi^{-2}|\Delta\psi|^2
\\ \noalign{\medskip}\displaystyle
+s^{-1}\int_Qe^{-2s\alpha^*}(\xi^*)^{-1}|P_1\psi|^2
+s^{-4}\int_{Q}e^{-2s\alpha}\xi^{-4}|\nabla\psi_t|^2\lesssim B_3+B_4
\\ \noalign{\medskip}\displaystyle
+s^{15}\int_{Q_{\cal O}}e^{-2s\alpha}\xi^{15}|\varphi|^2
+s^{-1}\int_Qe^{-2s\alpha}\xi^{-1}(|g_0|^2+|\nabla g_0|^2)
+s^{-2}\int_Qe^{-2s\alpha}\xi^{-2}(|g_1|^2+|\nabla g_1|^2).
\end{array}
\end{equation}

\subsubsection{Further estimates on $\varphi$ and $\psi$}

Let $\theta_0(t):=s^{-3/2-1/m}e^{-s\alpha^*(t)}(\xi^*(t))^{-3/2-2/m}$. Then,
$$
(\varphi^*,\pi^*)(t,x)=\theta_0(T-t)(\varphi,\pi)(T-t,x)
$$
satisfies system (\ref{Barcelona}) with
$$
h(t,x):=\theta_0(T-t)(P_1\psi)(T-t,x)
$$
and
$$
h_V(t,x):=\theta_0'(T-t)\varphi(T-t,x)+\theta_0(T-t)g_0(T-t,x).
$$
Using Lemma~\ref{estimationL} (c), we have
$$
\|\varphi^*\|^2_{L^2(0,T;H^4(\Omega))\cap H^1(0,T;H^2(\Omega))}\lesssim \|h\|^2_{L^2(0,T;H^2(\Omega))\cap H^1(0,T;L^2(\Omega))}+\|h_V\|^2_{L^2(0,T;H^2(\Omega))}.
$$
Regarding the first term, one has
$$
\|h\|^2_{L^2(0,T;H^2(\Omega))\cap H^1(0,T;L^2(\Omega))}\lesssim \|\theta_0\psi\|^2_{L^2(0,T;H^3(\Omega))}+\|\theta_0'P_1\psi\|^2_{L^2(Q)}+
\|\theta_0P_1\psi_t\|^2_{L^2(Q)}.
$$
Regarding the second term, we deduce
$$
\|h_V\|^2_{L^2(0,T;H^2(\Omega))}\lesssim \|\theta_0'\Delta\varphi\|^2_{L^2(Q)}+\|\theta_0\Delta g_0\|^2_{L^2(Q)}.
$$
Consequently, we infer
\begin{equation}\label{Steaua}
\begin{array}{l}\displaystyle
\|\varphi^*\|^2_{L^2(0,T;H^4(\Omega))\cap H^1(0,T;H^2(\Omega))}
\\ \noalign{\medskip}\displaystyle
\lesssim
\|\theta_0\psi\|^2_{L^2(0,T;H^3(\Omega))}+\|\theta_0'P_1\psi\|^2_{L^2(Q)}+
\|\theta_0P_1\psi_t\|^2_{L^2(Q)}
+ \|\theta_0'\Delta\varphi\|^2_{L^2(Q)}+\|\theta_0\Delta g_0\|^2_{L^2(Q)}.
\end{array}
\end{equation}
Let $\theta_1(t):=s^{-2-1/m}e^{-s\alpha^*(t)}(\xi^*(t))^{-2-1/m}$. The function
$$
\psi^*(t,x)=\theta_1(T-t)\psi(T-t,x)
$$
satisfies system (\ref{Espanyol}) with
$$
h(t,x):=\theta_1(T-t)(\nabla\times\varphi)(T-t,x)
$$
and
$$
h_0(t,x):=\theta_1'(T-t)\psi(T-t,x)+\theta_1(T-t)g_1(T-t,x).
$$
Using Lemma~\ref{estimationM} (b), we have
$$
\|\psi^*\|^2_{L^2(0,T;H^4(\Omega))}\lesssim \|h\|^2_{L^2(0,T;H^2(\Omega))\cap H^1(0,T;L^2(\Omega))}+\|h_0\|^2_{L^2(0,T;H^2(\Omega))}.
$$
Regarding the first term, one has
$$
\|h\|^2_{L^2(0,T;H^2(\Omega))\cap H^1(0,T;L^2(\Omega))}\lesssim \|\theta_1\varphi\|^2_{L^2(0,T;H^3(\Omega))}+\|\theta_1'\nabla\times\varphi\|^2_{L^2(Q)}+
\|\theta_1\nabla\times\varphi_t\|^2_{L^2(Q)}.
$$
Regarding the second term, we deduce
$$
\|h_0\|^2_{L^2(0,T;H^2(\Omega))}\lesssim \|\theta_1'\Delta\psi\|^2_{L^2(Q)}+\|\theta_1\Delta g_1\|^2_{L^2(Q)}.
$$
Consequently, we infer
$$
\|\psi^*\|^2_{L^2(0,T;H^4(\Omega))}\lesssim
\|\theta_1\varphi\|^2_{L^2(0,T;H^3(\Omega))}+\|\theta_1'\nabla\times\varphi\|^2_{L^2(Q)}+
\|\theta_1\nabla\times\varphi_t\|^2_{L^2(Q)}
+ \|\theta_1'\Delta\psi\|^2_{L^2(Q)}+\|\theta_1\Delta g_1\|^2_{L^2(Q)}.
$$
Putting this together with (\ref{Steaua}), we deduce
\begin{equation}\label{Nadal}
\begin{array}{l}\displaystyle
\|\varphi^*\|^2_{L^2(0,T;H^4(\Omega))\cap H^1(0,T;H^2(\Omega))}
+\|\psi^*\|^2_{L^2(0,T;H^4(\Omega))\cap H^1(0,T;H^2(\Omega))}
\\ \noalign{\medskip}\displaystyle
\lesssim
\|\theta_0\psi\|^2_{L^2(0,T;H^3(\Omega))}+\|\theta_0'P_1\psi\|^2_{L^2(Q)}+
\|\theta_0P_1\psi_t\|^2_{L^2(Q)}
+ \|\theta_0'\Delta\varphi\|^2_{L^2(Q)}+\|\theta_0\Delta g_0\|^2_{L^2(Q)}
\\ \noalign{\medskip}\displaystyle
+\|\theta_1\varphi\|^2_{L^2(0,T;H^3(\Omega))}+\|\theta_1'\nabla\times\varphi\|^2_{L^2(Q)}+
\|\theta_1\nabla\times\varphi_t\|^2_{L^2(Q)}
+ \|\theta_1'\Delta\psi\|^2_{L^2(Q)}+\|\theta_1\Delta g_1\|^2_{L^2(Q)}.
\end{array}
\end{equation}
We now estimate the terms in the right-hand side of (\ref{Nadal}) concerning $\varphi$ and $\psi$ with the help of (\ref{Carlemanvarpsi5}).
\begin{itemize}
\item First, we observe that from the definitions of $\theta_0$ and $\theta_1$, the term $\|\theta_1\varphi\|^2_{L^2(0,T;H^3(\Omega))}$ is absorbed by the left-hand side of (\ref{Nadal}).

\item Next, since
$$
|\theta_0'|\lesssim s^{-1/2-1/m}e^{-s\alpha^*}(\xi^*)^{-1/2-1/m},
$$
the terms $\|\theta_0'P_1\psi\|^2_{L^2(Q)}$ and $\|\theta_0'\Delta\varphi\|^2_{L^2(Q)}$ are absorbed by the fourth and second terms in the left-hand side of estimate (\ref{Carlemanvarpsi5}), respectively. Moreover, using that
$$
|\theta_1'|\lesssim s^{-1-1/m}e^{-s\alpha^*}(\xi^*)^{-1},
$$
the terms $\|\theta_1'\nabla\times\varphi\|^2_{L^2(Q)}$ and $\|\theta_1'\Delta\psi\|^2_{L^2(Q)}$ are absorbed by the second and third terms in the left-hand side of estimate (\ref{Carlemanvarpsi5}) respectively, provided that $s$ is large enough.

\item Then, observe that
\begin{equation}\label{Allemagne}
\int_Q\theta_0^2|\Delta\varphi_t|^2\leq 2\left(\|\varphi^*_t\|^2_{L^2(0,T;H^2(\Omega))}+\|\theta_0'\Delta\varphi\|^2_{L^2(Q)}\right)
\end{equation}
so $\|\theta_1\nabla\times\varphi_t\|^2_{L^2(Q)}$ is absorbed by the left-hand sides of (\ref{Carlemanvarpsi5}) and (\ref{Nadal}).

\item Similarly, we have
\begin{equation}\label{Argentine}
\int_Q\theta_1^2|\Delta\psi_t|^2\leq 2\left(\|\psi^*_t\|^2_{L^2(0,T;H^2(\Omega))}+\|\theta_1'\Delta\psi\|^2_{L^2(Q)}\right).
\end{equation}
Moreover, using the equation of $\psi$, we get that
$$
s^{-2}\int_Q e^{-2s\alpha^*}(\xi^*)^{-2}|\psi_t|^2
$$
is estimated by the left-hand side of (\ref{Carlemanvarpsi5}). Integrating by parts in space, we have that
$$
s^{-3-2/m}\int_Qe^{-2s\alpha^*}(\xi^*)^{-3-4/m}|\nabla\psi_t|^2\lesssim s^{-2}\int_Q e^{-2s\alpha^*}(\xi^*)^{-2}|\psi_t|^2+\int_Q\theta_0^2|\Delta\psi_t|^2,
$$
so that, from (\ref{Argentine}), we can absorb the term $\|\theta_0P_1\psi_t\|^2_{L^2(Q)}$ by the left-hand sides of (\ref{Carlemanvarpsi5}) and (\ref{Nadal}).

\item Finally, from an interpolation argument, we find that
$$
\begin{array}{l}\displaystyle
\|\theta_0\psi\|^2_{L^2(0,T;H^3(\Omega))}=s^{-3-2/m}\int_0^Te^{-2s\alpha^*}(\xi^*)^{-3-4/m}\|\psi\|_{H^3(\Omega)}^2
\\ \noalign{\medskip}\displaystyle
\phantom{\|\theta_0\psi\|^2_{L^2(0,T;H^3(\Omega))}}
\leq \varepsilon\|\theta_1\psi\|^2_{L^2(0,T;H^4(\Omega))}+Cs^{-2-2/m}\int_0^Te^{-2s\alpha^*}(\xi^*)^{-2-6/m}\|\psi\|_{H^2(\Omega)}^2
\end{array}
$$
for some $C>0$ (which might depend on $\varepsilon>0$).
\end{itemize}

We conclude that
\begin{equation}\label{Carlemanvarpsi6}
\begin{array}{l}\displaystyle
s^2\int_Qe^{-2\alpha}\xi^2|\varphi|^2+s^{-1}\int_Qe^{-2s\alpha}\xi^{-1}|\Delta\varphi|^2
+s^{-2}\int_Qe^{-2s\alpha}\xi^{-2}|\Delta\psi|^2
\\ \noalign{\medskip}\displaystyle
+s^{-1}\int_Qe^{-2s\alpha^*}(\xi^*)^{-1}|P_1\psi|^2
+s^{-4}\int_{Q}e^{-2s\alpha}\xi^{-4}|\nabla\psi_t|^2
\\ \noalign{\medskip}\displaystyle
+\|\varphi^*\|^2_{L^2(0,T;H^4(\Omega))\cap H^1(0,T;H^2(\Omega))}
+\|\psi^*\|^2_{L^2(0,T;H^4(\Omega))\cap H^1(0,T;H^2(\Omega))}
\\ \noalign{\medskip}\displaystyle
\lesssim B_3+B_4+s^{15}\int_{Q_{\cal O}}e^{-2s\alpha}\xi^{15}|\varphi|^2
+s^{-1}\int_Qe^{-2s\alpha}\xi^{-1}(|g_0|^2+|\nabla g_0|^2+|\nabla^2g_0|^2)
\\ \noalign{\medskip}\displaystyle
+s^{-2}\int_Qe^{-2s\alpha}\xi^{-2}(|g_1|^2+|\nabla g_1|^2+|\nabla^2g_1|^2).
\end{array}
\end{equation}

\subsubsection{Estimate of the boundary terms in (\ref{Carlemanvarpsi6})}

We first establish a useful trace lemma.
\begin{lemma}\label{Ferrero}
There exists $C>0$ such that
$$
\int_{\partial\Omega}\left|\frac{\partial u}{\partial\nu}\right|^2\leq C\|u\|_{H^2(\Omega)}^{3/2}\|u\|_{L^2(\Omega)}^{1/2},
$$
for all $u\in H^2(\Omega)$.
\end{lemma}
\begin{proof}
Let $\kappa\in C^2(\overline\Omega)$ be a function satisfying
$$
\frac{\partial\kappa}{\partial\nu}=1\quad\hbox{and}\quad \kappa=1\quad\hbox{ on }\partial\Omega.
$$
Integrating by parts, we have
$$
\int_{\Omega}(\nabla\kappa\cdot\nabla u)\Delta u=\int_{\partial\Omega}\left|\frac{\partial u}{\partial\nu}\right|^2-\int_{\Omega}\nabla(\nabla\kappa\cdot\nabla u)\cdot\nabla u.
$$
Using now
$$
\|\nabla u\|^2_{L^2(\Omega)}\lesssim\|u\|_{L^2(\Omega)}\|u\|_{H^2(\Omega)}
$$
along with Cauchy-Schwarz inequality, the proof is complete.
\end{proof}


Using Lemma \ref{Ferrero}, we find that
$$
\begin{array}{l}\displaystyle
B_3:=s^{-3}\int_{\Sigma}e^{-2s\alpha^*}(\xi^*)^{-3}\left|\frac{\partial \Delta\varphi}{\partial\nu}\right|^2
\\ \noalign{\medskip}\displaystyle
\phantom{B_3:}
\leq \varepsilon \|\theta_0\varphi\|^2_{L^2(0,T;H^4(\Omega))}
+C_{\varepsilon} s^{-3+6/m}\int_0^Te^{-2s\alpha^*}(\xi^*)^{-3+12/m}\|\varphi\|^2_{H^2(\Omega)}
\end{array}
$$
and
$$
\begin{array}{l}\displaystyle
B_4:=s^{-4}\int_{\Sigma}e^{-2s\alpha^*}(\xi^*)^{-4}\left|\frac{\partial \Delta\psi}{\partial\nu}\right|^2
\\ \noalign{\medskip}\displaystyle
\phantom{B_4:}
\leq \varepsilon \|\theta_1\psi\|^2_{L^2(0,T;H^4(\Omega))}
+C_{\varepsilon} s^{-4+6/m}\int_0^Te^{-2s\alpha^*}(\xi^*)^{-4+6/m}\|\psi\|^2_{H^2(\Omega)}.
\end{array}
$$
Using that $m\geq 6$, these two terms are absorbed by the left-hand side of (\ref{Carlemanvarpsi6}).



This ends the proof of Proposition \ref{Carleman2}.

\section{Proof of Theorems~\ref{Control3} and \ref{Control2}}


\subsection{Observability inequality and controllability of a linear problem}\label{section2.1}


In this paragraph we prove the null controllability of the following linear system :
\begin{equation}\label{linear}
\left\{\begin{array}{ll}
Ly+\nabla p=P_1\omega+f_0+(3-d)\mathds{1}_{\cal O}u,\quad \nabla\cdot y=0&\hbox{in }Q,
\\ \noalign{\medskip}
M\omega+(y\cdot\nabla)\overline\omega=\nabla\times y+(d-2)\mathds{1}_{\cal O}v+f_1&\hbox{in }Q,
\\ \noalign{\medskip}
y=\omega=0&\hbox{on }\Sigma,
\\ \noalign{\medskip}
y(0,\cdot)=y_0,\quad \omega(0,\cdot)=\omega_0&\hbox{in }\Omega,
\end{array}\right.
\end{equation}
for suitable $f_0$ and $f_1$, $y_0\in V$ and $\omega_0\in H^2(\Omega)\cap H^1_0(\Omega)$. Here, we have denoted
\begin{equation}\label{LM2}
Ly:=y_t-\Delta y\quad\hbox{and}\quad M\omega:=\omega_t-\Delta\omega-(d-2)\nabla(\nabla\cdot\omega).
\end{equation}

Before proving this result we need to prove a new Carleman estimate with weight functions only vanishing at $t=T$. Let
\begin{equation}\label{beta}
\beta(t,x)=\frac{e^{2\lambda\|\eta\|_{\infty}}-e^{\lambda\eta(x)}}{\widetilde\ell(t)^8},\quad\gamma (t,x)=\frac{e^{\lambda\eta(x)}}{\widetilde\ell(t)^8},
\end{equation}
where $\widetilde\ell$ is the $C^{\infty}([0,T])$ function given by
$$
\widetilde\ell(t)=
\left\{\begin{array}{ll}
\ell(T/2)&\hbox{ for }t\in [0,T/2],
\\ \noalign{\medskip}
\ell(t)&\hbox{ for }t\in [T/2,T].
\end{array}\right.
$$

\subsubsection{Three-dimensional case}
We will prove the following result :
\begin{proposition}
Under the same assumptions of Proposition \ref{Carleman3}, there exists $C>0$ such that the solutions of (\ref{adjoint}) satisfy
\begin{equation}\label{CI2}
\begin{array}{l}\displaystyle
\int_Q e^{-2s\beta} \gamma^2 |\psi|^2+\int_Q e^{-2s\beta} |\varphi|^2+\int_{\Omega}|\varphi(0,\cdot)|^2+\int_{\Omega}|\psi(0,\cdot)|^2
\\ \noalign{\medskip}\displaystyle
\hskip2cm\leq C\left(\int_Q e^{-2s \beta} \gamma^{-3}\left( |g_0|^2+|\nabla g_0|^2\right)+ \int_Qe^{-2s\beta} |g_1|^2 + \int_{Q_{\cal O}} e^{-2s\beta}\gamma^{4} |\psi|^2 \right).
\end{array}
\end{equation}
\end{proposition}
\begin{proof} To prove estimate (\ref{CI2}) we start by observing that, since $\beta=\alpha$ in $(T/2,T)\times\Omega$ and $\beta\leq\alpha$,
\begin{equation}\label{plate}
\begin{array}{l}\displaystyle
\int_{(T/2,T)\times\Omega}e^{-2s\beta} \gamma^2 |\psi|^2+\int_{(T/2,T)\times\Omega} e^{-2s\beta} |\varphi|^2
\leq\int_{Q}e^{-2s\alpha} \xi^2 |\psi|^2+\int_{Q} e^{-2s\alpha} |\varphi|^2
\\ \noalign{\medskip}\displaystyle
\leq C\left(s^{-3} \int_Q e^{-2s \alpha} \xi^{-3}\left( |g_0|^2+|\nabla g_0|^2\right)+ \int_Qe^{-2s\alpha} |g_1|^2 + s^4 \int_{Q_{\cal O}} e^{-2s\alpha}\xi^{4} |\psi|^2 \right)
\\ \noalign{\medskip}\displaystyle
\leq C\left(s^{-3} \int_Q e^{-2s \beta} \gamma^{-3}\left( |g_0|^2+|\nabla g_0|^2\right)+ \int_Qe^{-2s\beta} |g_1|^2 + s^4 \int_{Q_{\cal O}} e^{-2s\beta}\gamma^{4} |\psi|^2 \right).
\end{array}
\end{equation}
Here, we have also used the Carleman inequality (\ref{CI}) and the fact that
$$
e^{-2s\alpha}\xi^4\lesssim e^{-2s\beta}\gamma^4.
$$
In order to perform an estimate on $(T/2,T)\times\Omega$, we introduce $\sigma_2\in C^1([0,T])$ satisfying $\sigma_2(t)=1$ for $t\in [0,T/2]$ and $\sigma_2(t)=0$ for $t\in [3T/4,T]$. Then, $\sigma_2(\varphi,\pi,\psi)$ satisfies
$$
\left\{
\begin{array}{llll}
     -(\sigma_2\varphi)_t-\Delta (\sigma_2\varphi) +\nabla (\sigma_2\pi)&=&\nabla \times (\sigma_2\psi)+(\nabla (\sigma_2\psi))^T  \overline{\omega}+\sigma_2g_0-\sigma_2'\varphi \qquad  &\text{ in } Q,  \\
    -(\sigma_2\psi)_t-\Delta (\sigma_2\psi) - \nabla(\nabla \cdot (\sigma_2\psi))&=& \nabla \times (\sigma_2\varphi) +\sigma_2g_1-\sigma_2'\psi \qquad &\text{ in } Q,\\
 \nabla \cdot (\sigma_2\varphi) &=& 0\qquad &\text{ in } Q,\\
     \sigma_2\varphi &=& 0 \qquad  &\text{ on } \Sigma, \\
     \sigma_2\psi&=&0  \qquad &\text{ on } \Sigma,
     \\
     (\sigma_2\varphi)(T,\cdot)&=&0  \qquad &\text{ in } \Omega,
     \\
     (\sigma_2\psi)(T,\cdot)&=&0  \qquad &\text{ in } \Omega,
\end{array}
\right.
$$
(see  (\ref{adjoint})). For this system, we have (see (\cite{L}))
$$
\int_Q|\sigma_2\varphi|^2+\int_Q|\sigma_2\psi|^2+\int_{\Omega}|\varphi(0,\cdot)|^2+\int_{\Omega}|\psi(0,\cdot)|^2\leq C\left(\int_{(0,3T/4)\times\Omega}(|g_0|^2+|g_1|^2)+\int_{(T/2,3T/4)\times\Omega}(|\varphi|^2+|\psi|^2)\right).
$$
Observing now that $e^{-2s\beta}\geq C$ in $(0,3T/4)\times\Omega$, $e^{-2s\alpha}\geq C$ in $(T/2,3T/4)\times\Omega$ and using again the Carleman inequality (\ref{CI}), we deduce in particular
$$
\begin{array}{l}\displaystyle
\int_{(0,T/2)\times\Omega}e^{-2s\beta}|\varphi|^2+\int_{(0,T/2)\times\Omega}e^{-2s\beta} \gamma^2|\psi|^2+\int_{\Omega}|\varphi(0,\cdot)|^2+\int_{\Omega}|\psi(0,\cdot)|^2
\\ \noalign{\medskip}\displaystyle
 \leq C\left(s^{-3} \int_Q e^{-2s \beta} \gamma^{-3}\left( |g_0|^2+|\nabla g_0|^2\right)+ \int_Qe^{-2s\beta} |g_1|^2 + s^4 \int_{Q_{\cal O}} e^{-2s\beta}\gamma^{4} |\psi|^2 \right).
 \end{array}
$$
Combining this with (\ref{plate}), we deduce the desired inequality (\ref{CI2}).
\end{proof}
\begin{remark}
If we denote
$$
\widehat\beta(t):=\min_{x\in\overline\Omega}\beta(t,x),\quad\widehat\gamma (t):=\min_{x\in\overline\Omega}\gamma(t,x),
$$
then, we deduce from (\ref{CI2}):
\begin{equation}\label{CI3}
\begin{array}{l}\displaystyle
\|\gamma e^{-s\beta}\psi\|^2_{L^2(Q)}+\|e^{-s\beta} \varphi\|^2_{L^2(Q)}+\|\varphi(0,\cdot)\|^2_{L^2(\Omega)}+\|\psi(0,\cdot)\|^2_{L^2(\Omega)}
\\ \noalign{\medskip}\displaystyle
\hskip2cm\leq C(\|e^{-s\widehat\beta} \widehat\gamma^{-3/2}g_0\|^2_{L^2(V)}+
\|e^{-s\beta}g_1\|^2_{L^2(Q)} + \|e^{-s\beta}\gamma^2\psi\|^2_{L^2(Q_{\cal O})} ).
\end{array}
\end{equation}
\end{remark}

\vskip0.7cm Now, we are ready to
solve the null controllability problem for the linear system
(\ref{linear}). For simplicity, we introduce the following
weight functions:
$$
\rho_0(t,x):=e^{s\beta(t,x)},\quad
\rho_1(t,x):=e^{s\beta(t,x)}\gamma(t,x)^{-1},\quad \rho_2(t,x):=e^{s\beta(t,x)}\gamma(t,x)^{-2},\quad \rho_3(t):=e^{s\widehat\beta(t)}\widehat\gamma(t)^{3/2}.
$$

The null controllability of system (\ref{linear}) will be established in some weighted spaces which we present now :
\begin{equation}\label{E1}
E_1:=\{(y,p,v,\omega)
\in E_0 : \rho_0(Ly+\nabla p-\nabla\times\omega)\in
L^2(Q),\,\rho_1(M\omega+(y\cdot\nabla)\overline\omega-\nabla\times y-\mathds{1}_{\cal O}v)\in L^2(Q)\}
\end{equation}
where
$$
\displaystyle E_0=\{(y,p,v,\omega): (\rho_3)^{3/4}y\in L^2(0,T;H^2(\Omega))\cap L^{\infty}(0,T;V),\,(\rho_0)^{3/4} \omega\in L^2(0,T;H^2(\Omega)),\,\rho_2v\in L^2(Q)\}.
$$

Of course, $E_1$ and $E_0$ are Banach spaces for the norms
$$
\|(y,p,v,\omega)\|_{E_0}=(\|(\rho_3)^{3/4}y\|_{L^2(0,T;H^2(\Omega))\cap L^{\infty}(0,T;V)}^2+\|(\rho_0)^{3/4} \omega\|^2_{L^2(0,T;H^2(\Omega))}+\|\rho_2v\|^2_{L^2(Q)})^{1/2}
$$
and
$$
\begin{array}{l}
\|(y,p,v,\omega)\|_{E_1}=(\|(y,p,v,\omega)\|_{E_0}^2+
\|\rho_0(Ly+\nabla p-\nabla\times\omega)\|^2_{L^2(Q)}
\\ \noalign{\medskip}
\phantom{\|(y,p,v,\omega)\|_{E_1}}
+\|\rho_1(M\omega+(y\cdot\nabla)\overline\omega-\nabla\times y-\mathds{1}_{\cal O}v)\|^2_{L^2(Q)})^{1/2}.
\end{array}
$$

Then, we have the following result:
\begin{proposition}\label{prop2}
   Let us assume that $\overline\omega\in L^{\infty}(0,T;W^{1,3+\delta}(\Omega))\cap H^1(0,T;L^3(\Omega))$ for some $\delta>0$, $y_0\in V$, $\omega_0\in H^2(\Omega)\cap H^1_0(\Omega)$,
$\rho_0f_0\in L^2(Q)$ and $\rho_1f_1\in L^2(Q)$. Then, there exists a control $v\in L^2(Q)$ such that,
if $(y,\omega)$ is (together with some $p$) the associated solution
to (\ref{linear}), one has $(y,p,v,\omega)\in E_1$. In particular, $y(T)=\omega(T)=0$ in
$\Omega$.
\end{proposition}

\begin{proof}
Let us introduce the space
$$
\begin{array}{l}
P_0=\{(\varphi,\pi,\psi)\in C^3(\overline Q):\nabla\cdot\varphi=0 \hbox{ in } Q,\,\varphi=\psi=0 \hbox{ on } \Sigma,
\\ \noalign{\medskip}
(L^*\varphi+\nabla\pi-\nabla \times \psi-(\nabla \psi)^T \overline{\omega})=0\hbox{ on } \Sigma,\,
\nabla\cdot(L^*\varphi+\nabla\pi-\nabla \times \psi-(\nabla \psi)^T \overline{\omega})=0\hbox{ in } Q\}
\end{array}
$$
and consider the bilinear form
\begin{equation}\label{a}
\begin{array}{l}\displaystyle
a_0((\widehat \varphi,\widehat \pi,\widehat\psi),(\varphi,\pi,\psi)):=\int_Q(\rho_0)^{-2}(M^*\widehat\psi-\nabla\times\widehat\varphi)\cdot(M^*\psi-\nabla\times\varphi)
+\int_{\cal O}(\rho_2)^{-2}\widehat\psi\cdot\psi
\\ \noalign{\medskip}\displaystyle
+\int_Q(\rho_3)^{-2}(L^*\widehat\varphi+\nabla\widehat\pi-\nabla \times\widehat \psi-(\nabla \widehat\psi)^T\overline\omega) \cdot (L^*\varphi+\nabla\pi-\nabla \times \psi-(\nabla \psi)^T \overline{\omega})
\\ \noalign{\medskip}\displaystyle
+\int_Q(\rho_3)^{-2}\nabla(L^*\widehat\varphi+\nabla\widehat\pi-\nabla \times\widehat \psi-(\nabla \widehat\psi)^T\overline\omega): \nabla(L^*\varphi+\nabla\pi-\nabla \times \psi-(\nabla \psi)^T \overline{\omega}).
\end{array}
\end{equation}
Here, $L^*$ and $M^*$ denote the formal adjoint operators of $L$ and $M$ respectively :
$$
L^*\varphi:=-\varphi_t-\Delta \varphi\quad\hbox{and}\quad M^*\psi:=-\psi_t-\Delta\psi-(d-2)\nabla(\nabla\cdot\psi).
$$

From the Carleman inequality (\ref{CI3}), this bilinear form is an inner product in $P_0$. We consider the Hilbert space resulting of the completion of $P_0$ with $a(\cdot,\cdot)$ and we call it $P$.

We introduce now the linear form $b_0:P\rightarrow \mathbb{R}$:
$$
b_0(\varphi,\pi,\psi):= \int_Q f_0\cdot\varphi+\int_Q f_1\cdot\psi+\int_{\Omega}\varphi(0,x)\cdot y_0(x)
+\int_{\Omega}\psi(0,x)\cdot\omega_0(x).
$$
Then, in virtue of the Carleman inequality (\ref{CI3}) this linear form is continuous. Consequently, from the Lax-Milgram's Lemma there exists a unique solution $(\widehat \varphi,\widehat \pi,\widehat\psi)\in P$ of
\begin{equation}\label{LM}
a_0((\widehat \varphi,\widehat \pi,\widehat\psi),(\varphi,\pi,\psi))=b_0(\varphi,\pi,\psi)   \quad\forall (\varphi,\pi,\psi)\in P.
\end{equation}

Let us now define the following quantities :
$$
\widehat y:=(\rho_3)^{-2}[(L^*\widehat\varphi+\nabla\widehat\pi-\nabla \times\widehat \psi-(\nabla \widehat\psi)^T\overline\omega)
-\Delta(L^*\widehat\varphi+\nabla\widehat\pi-\nabla \times\widehat \psi-(\nabla \widehat\psi)^T\overline\omega)],
$$
$$
\widehat\omega:=(\rho_0)^{-2}(M^*\widehat\psi-\nabla\times\widehat\varphi)
$$
and
$$
\widehat v:=-(\rho_2)^{-2}\widehat\psi.
$$
Then, from (\ref{a}) we readily have
$$
\|\rho_3\widehat y\|_{L^2(V')}+\|\rho_0\widehat\omega\|_{L^2(Q)}+\|\rho_2\widehat v\|_{L^2(Q)}=a_0((\widehat \varphi,\widehat \pi,\widehat\psi),
(\widehat \varphi,\widehat \pi,\widehat\psi))<+\infty.
$$

We consider now the weak solution $(\tilde y,\tilde p,\tilde\omega)$ of system (\ref{linear}) with $v:=\widehat v$ and its adjoint system
\begin{equation}\label{adjoint2}
\left\{
\begin{array}{llll}
     L^* \varphi +\nabla \pi&=&\nabla \times \psi+(\nabla \psi)^T  \overline{\omega}+g_0 \qquad  &\text{ in } Q,  \\
     M^* \psi&=& \nabla \times \varphi +g_1 \qquad &\text{ in } Q,\\
 \nabla \cdot \varphi &=& 0\qquad &\text{ in } Q,\\
     \varphi &=& 0 \qquad  &\text{ on } \Sigma, \\
     \psi&=&0  \qquad &\text{ on } \Sigma,\\
     \varphi(T,\cdot) &=& 0 \qquad  &\text{ in } \Omega, \\
     \psi(T,\cdot)&=&0  \qquad &\text{ in } \Omega.
\end{array}
\right.
\end{equation}
We multiply the equation of $\tilde y$ by $\varphi$, the equation of $\tilde\omega$ by $\psi$ and we integrate by parts. We obtain
$$
\begin{array}{l}\displaystyle
\int_Q\tilde y\cdot g_0+\int_Q\tilde\omega\cdot g_1
=\int_Qf_0\cdot\varphi
+\int_Q(f_1+\mathds{1}_{\cal O}\widehat v)\cdot\psi+\int_{\Omega}y_0\cdot\varphi_{|t=0}+\int_{\Omega}\omega_0\cdot\psi_{|t=0},
\end{array}
$$
for all $g_0,\,g_1\in L^2(Q)$. That is to say, $(\tilde y,\tilde p,\tilde\omega)$ is also the solution by transposition of (\ref{linear}) with $v=\widehat v$.

Then, from (\ref{LM}), it is not difficult to see that $(\widehat y,\widehat p,\widehat\omega)=(\tilde y,\tilde p,\tilde\omega)$. Moreover, one can perform
regularity estimates for our system in order to prove that the weak solution of (\ref{linear}) with $v=\widehat v$ satisfies $(\tilde y,\tilde p,\widehat v,\tilde\omega)\in E_1$. For all the details, one can see for instance the proof of Proposition 4.3 in \cite{CG}.
\end{proof}

\subsubsection{Two-dimensional case}



In this case, we can prove the following result :
\begin{proposition}
Let $d=2$, $m=8$ and $T>0$. Then, under the same assumptions of Proposition \ref{Carleman2} there exists $C>0$ such that the solutions of (\ref{adjoint}) satisfy
\begin{equation}\label{CI4}
\begin{array}{l}
\| e^{-s\beta^*}(\gamma^*)^{-1/2}\varphi\|^2_{L^2(0,T;H^2(\Omega))}+\|e^{-s\beta^*}(\gamma^*)^{-1} \psi\|^2_{L^2(0,T;H^2(\Omega))}+\|\varphi(0,\cdot)\|^2_{L^2(\Omega)}+\|\psi(0,\cdot)\|^2_{L^2(\Omega)}
\\ \noalign{\medskip}\displaystyle
\hskip2cm\leq C(\|e^{-s\widehat\beta} \widehat\gamma^{-1}g_0\|^2_{L^2(0,T;H^2(\Omega))}+
\|e^{-s\widehat\beta} \widehat\gamma^{-1}g_1\|^2_{L^2(0,T;H^2(\Omega))}+  \|e^{-s\beta}\gamma^{15/2}\psi\|^2_{L^2(Q_{\cal O})} ),
\end{array}
\end{equation}
where we have denoted
$$
\beta^*(t):=\max_{x\in\overline\Omega}\beta(t,x),\quad\gamma^* (t):=\max_{x\in\overline\Omega}\gamma(t,x),
$$
and the other weights were defined at the beginning of paragraph~$\ref{section2.1}$.
\end{proposition}
The proof follows from the Carleman estimate stated in Proposition~\ref{Carleman2}.

\begin{remark} Let us define
$$
L^*_H :=-\partial_t-{\cal P}_L\circ \Delta,
$$
where ${\cal P}_L:L^2(\Omega)\to H$ is the Leray projector (recall that the space $H$ is defined in (\ref{H})). If, in addition to the assumptions stated in Proposition~\ref{Carleman2}, one assumes that $\partial_tg_0,\,\partial_tg_1 \in L^2(Q)$, then,
using the classical regularization effect for the Stokes equation and for the heat equation, one can deduce from \eqref{CI4} the following Carleman inequality
\begin{equation}\label{CI5}
\begin{array}{l}
\| e^{-s\beta^*}(\gamma^*)^{-1/2}\varphi\|^2_{L^2(0,T;H^2(\Omega))}+\|e^{-s\beta^*}(\gamma^*)^{-1} \psi\|^2_{L^2(0,T;H^2(\Omega))}+\|\varphi(0,\cdot)\|^2_{L^2(\Omega)}+\|\psi(0,\cdot)\|^2_{L^2(\Omega)}
\\ \noalign{\medskip}\displaystyle
\hskip2cm\leq C(\|L_H^*(e^{-s\widehat\beta} \widehat\gamma^{-1}g_0)\|^2_{L^2(Q)}+
\|M^*(e^{-s\widehat\beta} \widehat\gamma^{-1}g_1)\|^2_{L^2(Q)}+  \|e^{-s\beta}\gamma^{15/2}\psi\|^2_{L^2(Q_{\cal O})} ).
\end{array}
\end{equation}
\end{remark}

\vskip0.7cm Now, we are ready to
solve the null controllability problem for the linear system
(\ref{linear}). For simplicity, we introduce the following
weight functions:
\begin{equation}\label{velib}
\varsigma_0(t):=e^{s\beta^*(t)}(\gamma^*(t))^{1/2},\quad
\varsigma_1(t):=e^{s\beta^*(t)}\gamma^{*}(t),\, \varsigma_2(t,x):=e^{s\beta(t,x)}\gamma(t,x)^{-15/2},\, \varsigma_3(t):=e^{s\widehat\beta(t)}\widehat\gamma(t).
\end{equation}

The null controllability of system (\ref{linear}) will be established in some weighted spaces which we present now :
\begin{equation}\label{F1}
F_1:=\{(y,p,u,\omega)
\in F_0 : \varsigma_0(Ly+\nabla p-P_1\omega-\mathds{1}_{\cal O}u)\in
L^2(Q),\,\varsigma_1(M\omega-\nabla\times y)\in L^2(Q),\varsigma_2u\in L^2(Q)\}
\end{equation}
where
$$
\displaystyle F_0=\{(y,p,u,\omega): (\varsigma_0)^{3/4}
y\in L^2(0,T;H^2(\Omega))\cap L^{\infty}(0,T;V),\,(\varsigma_1)^{3/4}\omega\in L^2(0,T;H^2(\Omega)\cap H^1_0(\Omega))\}.
$$
As in the three-dimensional case, these spaces are Banach spaces for the corresponding natural norms.

Then, we have the following result:
\begin{proposition}\label{prop22}
   Let $y_0\in V$, $\omega_0\in H^1_0(\Omega)$,
$\varsigma_0f_0\in L^2(Q)$ and $\varsigma_1f_1\in L^2(Q)$. Then, there exists a control $u\in L^2(Q)$ such that,
if $(y,\omega)$ is (together with some $p$) the associated solution
to (\ref{linear}), one has $(y,p,u,\omega)\in F_1$. In particular, $y(T, \cdot)=\omega(T, \cdot)=0$ in
$\Omega$.
\end{proposition}

\begin{proof}
Let us introduce the space
$$
\begin{array}{l}
B_0=\{(\varphi,\pi,\psi)\in C^{\infty}(\overline Q):\nabla\cdot\varphi=0 \hbox{ in } Q,\,\varphi=\psi=0 \hbox{ on } \Sigma,
\\ \noalign{\medskip}
\phantom{P_1\,}
L^*\varphi+\nabla\pi-P_1 \psi=0\hbox{ on } \Sigma,\,
\nabla\cdot(L^*\varphi+\nabla\pi-P_1 \psi)=0\hbox{ in } Q
\\ \noalign{\medskip}
\phantom{P_1\,} (L^*\varphi+\nabla\pi-P_1 \psi)(0,\cdot)=0\hbox{ in } \Omega,\,
M^*\psi-\nabla \times \varphi=0\hbox{ on } \Sigma,\,
\\ \noalign{\medskip}
\phantom{P_1\,} (M^*\psi-\nabla \times \varphi)(0,\cdot)=0\hbox{ in } \Omega
\}
\end{array}
$$
and consider the bilinear form
\begin{equation}\label{a}
\begin{array}{l}\displaystyle
a_1((\widehat \varphi,\widehat \pi,\widehat\psi),(\varphi,\pi,\psi)):=\int_Q M^*[(\varsigma_3)^{-1}(M^*\widehat\psi-\nabla\times\widehat\varphi)]\cdot M^*[(\varsigma_3)^{-1}(M^*\psi-\nabla\times\varphi)]
\\ \noalign{\medskip}\displaystyle
+\int_{Q_{\cal O}}(\varsigma_2)^{-2}\widehat\varphi\cdot\varphi+\int_QL_H^*[(\varsigma_3)^{-1}(L^*\widehat\varphi+\nabla\widehat\pi-P_1\widehat \psi)] \cdot L_H^*[(\varsigma_3)^{-1}(L^*\varphi+\nabla\pi-P_1 \psi)].
\end{array}
\end{equation}

From the Carleman inequality (\ref{CI5}), this bilinear form is an inner product in $B_0$. We consider the Hilbert space resulting of the completion of $B_0$ with $a_1(\cdot,\cdot)$ and we call it $\widetilde{B_0}$.

We introduce now the linear form $b_1:\widetilde{B_0}\rightarrow \mathbb{R}$:
$$
b_1(\varphi,\pi,\psi):= \int_Q f_0\cdot\varphi+\int_Q f_1\cdot\psi\,dx\,dt+\int_{\Omega}\varphi(0,\cdot)\cdot y_0
+\int_{\Omega}\psi(0,\cdot)\cdot\omega_0.
$$
Then, in virtue of the Carleman inequality (\ref{CI5}) this linear form is continuous. Consequently, from the Lax-Milgram's Lemma there exists a unique solution $(\widehat \varphi,\widehat \pi,\widehat\psi)\in \widetilde{B_0}$ of
\begin{equation}\label{LaxM}
a_1((\widehat \varphi,\widehat \pi,\widehat\psi),(\varphi,\pi,\psi))=b_1(\varphi,\pi,\psi)   \quad\forall (\varphi,\pi,\psi)\in\widetilde{B_0}.
\end{equation}

Let us now define the following quantities :
\begin{equation}\label{Windoff}
\widehat y:=L^*_H[(\varsigma_3)^{-1}(L^*\widehat\varphi+\nabla\widehat\pi-P_1\widehat \psi)],
\end{equation}
\begin{equation}\label{Windoff2}
\widehat\omega:=M^*[(\varsigma_3)^{-1}(M^*\widehat\psi-\nabla\times\widehat\varphi)]
\end{equation}
and
\begin{equation}\label{Windoff3}
\widehat u:=-(\varsigma_2)^{-2}\widehat\psi.
\end{equation}
Then, from (\ref{a}) and (\ref{LaxM}), we readily have
\begin{equation}\label{ikea}
\|\widehat y\|_{L^2(Q)}+\|\widehat\omega\|_{L^2(Q)}+\|\varsigma_2\widehat u\|_{L^2(Q_{\cal O})}\lesssim \|\varsigma_0f_0\|_{L^2(Q)}+
\|\varsigma_1f_1\|_{L^2(Q)}+\|y_0\|_{L^2(\Omega)}+\|\omega_0\|_{L^2(\Omega)}.
\end{equation}

We consider now the weak solution $(y_w,p_w,\omega_w)$ of system (\ref{linear}) with $u:=\widehat u$. We will show that
\begin{equation}\label{Bernard}
\left\{
\begin{array}{llll}
     (\varsigma_3)^{-1} M \widehat\omega &=&\omega_w\qquad  &\text{ in } Q,  \\
     \widehat \omega&=&0 \qquad &\text{ in } \Sigma,\\
     \widehat \omega_{|t=0}&=&0\qquad &\text{ in } \Omega,
 \end{array}
\right.
\end{equation}
and
\begin{equation}\label{Arnauld}
\left\{
\begin{array}{llll}
     (\varsigma_3)^{-1} L_H \widehat y&=&y_w, \ \nabla\cdot \widehat y=0\qquad  &\text{ in } Q,  \\
     \widehat y&=&0 \qquad &\text{ in } \Sigma,\\
     \widehat y_{|t=0}&=&0\quad &\text{ in } \Omega.
 \end{array}
\right.
\end{equation}
From \eqref{LaxM} and \eqref{Windoff}-\eqref{Windoff3}, we find
\begin{eqnarray*}
\int_Q \widehat \omega M^*[(\varsigma_3)^{-1}(M^*\psi-\nabla\times\varphi)] -\int_{Q_{\cal O}}\widehat u\cdot\varphi+\int_Q \widehat y \cdot L_H^*[(\varsigma_3)^{-1}(L^*\varphi+\nabla\pi-P_1 \psi)]\\
=\int_Q (L y_w+\nabla p_w-P_1\omega_w -\mathds{1}_{\cal O} \widehat u )\cdot \varphi +\int_{Q}(M \omega_w - \nabla \times y_w) \psi +
\int_\Omega \varphi(0,\cdot)\cdot y_0+\int_\Omega \psi(0,\cdot) \omega_0.
\end{eqnarray*}
Integrating by parts, one consequently gets
\begin{eqnarray*}
\int_Q \widehat \omega M^*[(\varsigma_3)^{-1}(M^*\psi-\nabla\times\varphi)] -\int_{Q_{\cal O}}\widehat u\cdot\varphi+\int_Q \widehat y \cdot L_H^*[(\varsigma_3)^{-1}(L^*\varphi+\nabla\pi-P_1 \psi)]\nonumber\\
=\int_Q y_w\cdot(L^* \varphi +\nabla \pi-P_1 \psi) +\int_{Q}\omega_w(M^* \psi - \nabla \times \varphi).
\end{eqnarray*}
Therefore, $(\widehat y, \widehat \omega)$ satisfies
$$
\int_Q \widehat \omega g_2+\int_Q \widehat y\cdot g_3=\int_Q\omega_w \Phi_2+\int_Q y_w\cdot \Phi_3
$$
for all $g_2\in L^2(Q)$ and all $g_3\in L^2(Q)$, where $(\Phi_2,\Phi_3)$ is the solution of
\begin{equation}
\left\{
\begin{array}{llll}
     M^*[(\varsigma_3)^{-1} \Phi_2]&=&g_2 \qquad  &\text{ in } Q,  \\
    L_H^*[(\varsigma_3)^{-1} \Phi_3]&=&g_3 \qquad &\text{ in } Q,\\
 \nabla \cdot \Phi_3 &=& 0\qquad &\text{ in } Q,\\
     \Phi_2 &=  0,  \qquad  \Phi_3&=0  \qquad  &\text{ on } \Sigma, \\
     ((\varsigma_3)^{-1}\Phi_2)(T,\cdot) &= 0,  \qquad ((\varsigma_3)^{-1}\Phi_3)(T,\cdot)&=0  \qquad&\text{ in } \Omega
\end{array}
\right.
\end{equation}
This weak formulation means exactly that $(\widehat y, \widehat \omega)$  satisfies \eqref{Bernard}-\eqref{Arnauld}.

\vskip0.5cm Let us prove now that
\begin{equation}\label{atele}
(\varsigma_0)^{3/4}y_w\in L^2(H^2)\cap L^{\infty}(H^1),\,(\varsigma_1)^{3/4}\omega_w\in L^2(H^2)\cap L^{\infty}(H^1).
\end{equation}

$\bullet$ Let us first prove that, up to some weight functions, $y_w$ and $\omega_w$ are in $L^2(Q)$. Indeed, let us define
$$
(y^*,p^*,\omega^*):=\theta_2(t)(y_w,p_w,\omega_w),
$$
where
$$
\theta_2(t):=(T-t)^{68}\varsigma_3(t).
$$
Then, $(y^*,p^*,\omega^*)$ satisfies
\begin{equation}\label{y*}
\left\{\begin{array}{ll}
Ly^*+\nabla p^*=P_1\omega^*+\theta_2(\mathds{1}_{\cal O}\widehat u+f_0)+(\theta_2)'y_w,\quad \nabla\cdot y^*=0&\hbox{in }Q,
\\ \noalign{\medskip}
M\omega^*=\nabla\times y^*+\theta_2f_1+(\theta_2)'\omega_w&\hbox{in }Q,
\\ \noalign{\medskip}
y^*=0,\quad\omega^*=0&\hbox{on }\Sigma,
\\ \noalign{\medskip}
y^*(0, \cdot)=\theta_2(0)y_0,\quad \omega^*(0, \cdot)=\theta_2(0)\omega_0&\hbox{in }\Omega.
\end{array}\right.
\end{equation}
We use now that $(y^*,p^*,\omega^*)$ is also the solution by transposition of (\ref{y*}) :
\begin{equation}\label{plusdatele}
\begin{array}{l}\displaystyle
\int_Qy^*\cdot h_0+\int_Q\omega^*\cdot h_1
=\int_Q\theta_2(f_0+\mathds{1}_{\cal O}\widehat u)\cdot\varphi
+\int_Q\theta_2f_1\psi+\int_Q(\theta_2)'y_w\cdot\varphi+\int_Q(\theta_2)'\omega_w\psi
\\ \noalign{\medskip}\displaystyle
\phantom{\int_Qy^*\cdot h_0+\int_Q\omega^*\cdot h_1}+\int_{\Omega}\theta_2(0)y_0\cdot\varphi(0,\cdot)+\int_{\Omega}\theta_2(0)\omega_0\psi(0,\cdot)
\end{array}
\end{equation}
for all $h_0,\,h_1\in L^2(Q)$, where $(\varphi,\pi,\psi)$ is the solution of
\begin{equation}\label{adjoint2}
\left\{
\begin{array}{llll}
     L^* \varphi +\nabla \pi&=&P_1 \psi+h_0 \qquad  &\text{ in } Q,  \\
     M^* \psi&=& \nabla \times \varphi +h_1 \qquad &\text{ in } Q,\\
 \nabla \cdot \varphi &=& 0\qquad &\text{ in } Q,\\
     \varphi &=& 0 \qquad  &\text{ on } \Sigma, \\
     \psi&=&0  \qquad &\text{ on } \Sigma,\\
     \varphi(T,\cdot) &=& 0 \qquad  &\text{ in } \Omega, \\
     \psi(T,\cdot)&=&0  \qquad &\text{ in } \Omega.
\end{array}
\right.
\end{equation}
For this system, we have
\begin{equation}\label{kodak}
\|(\varphi,\psi)\|_{X_2}\lesssim \|h_0\|_{L^2(Q)}+\|h_1\|_{L^2(Q)}.
\end{equation}
where we have used the space $X_2:=L^2(0,T;H^2(\Omega))\cap L^{\infty}(0,T;H^1(\Omega))$ (endowed with its natural norm).

Observe that, from the definition of $\theta_2$ and $\varsigma_j\,(0\leq j\leq 3)$ (see (\ref{velib})), we have
$$
\begin{array}{l}\displaystyle
\left|
\int_Q\theta_2(f_0+\mathds{1}_{\cal O}\widehat u)\cdot\varphi
+\int_Q\theta_2f_1\psi+\int_{\Omega}\theta_2(0)y_0\cdot\varphi(0,\cdot)+\int_{\Omega}\theta_2(0)\omega_0\psi(0,\cdot)
\right|
\\ \noalign{\medskip}\displaystyle
\lesssim (\|\varsigma_0f_0\|_{L^2(Q)}+\|\varsigma_1f_1\|_{L^2(Q)}+\|\varsigma_0\widehat u\|_{L^2(Q_{\cal O})}+\|(y_0,\omega_0\|_{L^2(\Omega)})\|(\varphi,\psi)\|_{X_2}.
\end{array}
$$
Finally, using (\ref{Bernard})-(\ref{Arnauld}), we find
$$
\begin{array}{l}\displaystyle
\int_Q(\theta_2)'y_w\cdot\varphi+\int_Q(\theta_2)'\omega_w\psi=\int_Q(\theta_2)'(\varsigma_3)^{-1}L_H\widehat y\cdot\varphi+\int_Q(\theta_2)'(\varsigma_3)^{-1}M\widehat\omega\psi
\\ \noalign{\medskip}\displaystyle
=\int_QL_H^*((\theta_2)'(\varsigma_3)^{-1}\varphi)\cdot\widehat y+\int_QM^*((\theta_2)'(\varsigma_3)^{-1}\psi)\widehat\omega-\int_{\Omega}((\theta_2)'(\varsigma_3)^{-1})(0)(y_0\cdot\varphi(0,\cdot)+\omega_0\psi(0,\cdot))
\end{array}
$$
Using (\ref{ikea}) and the fact that $\|(\theta_2)'(\varsigma_3)^{-1}\|_{W^{1,\infty}(0,T)}\lesssim 1$, we obtain
$$
\left|\int_Q(\theta_2)'y_w\cdot\varphi+\int_Q(\theta_2)'\omega_w\psi\right|\lesssim (\|\varsigma_0f_0\|_{L^2(Q)}+\|\varsigma_1f_1\|_{L^2(Q)}+\|(y_0,\omega_0\|_{L^2(\Omega)})\|(\varphi,\psi)\|_{X_2}.
$$
Coming back to (\ref{plusdatele}) and using (\ref{kodak}), we deduce that $(y^*,\omega^*)\in L^2(Q)$ and
\begin{equation}\label{parisfc}
\|(y^*,\omega^*)\|_{L^2(Q)}\lesssim \|\varsigma_0f_0\|_{L^2(Q)}+\|\varsigma_1f_1\|_{L^2(Q)}+\|(y_0,\omega_0\|_{L^2(\Omega)}.
\end{equation}

 $\bullet$ Let us finally prove (\ref{atele}). To do so, we define
$$
(\widetilde y,\widetilde p,\widetilde\omega):=\theta_3(t)(y_w,p_w,\omega_w),
$$
where
$$
\theta_3(t):=(T-t)^{69}\varsigma_3(t).
$$
Then, $(\widetilde y,\widetilde p,\widetilde\omega)$ satisfies
\begin{equation}\label{y*}
\left\{\begin{array}{ll}
L\widetilde y+\nabla \widetilde p=P_1\widetilde\omega+\theta_3(\mathds{1}_{\cal O}\widehat u+f_0)+(\theta_3)'y_w,\quad \nabla\cdot \widetilde y=0&\hbox{in }Q,
\\ \noalign{\medskip}
M\widetilde\omega=\nabla\times \widetilde y+\theta_3 f_1+(\theta_3)'\omega_w&\hbox{in }Q,
\\ \noalign{\medskip}
\widetilde y=0,\quad\widetilde\omega=0&\hbox{on }\Sigma,
\\ \noalign{\medskip}
\widetilde y(0,\cdot)=\theta_3(0)y_0,\quad \widetilde\omega(0,\cdot)=\theta_3(0)\omega_0&\hbox{in }\Omega.
\end{array}\right.
\end{equation}
Using that $|(\theta_3)'|\lesssim \theta_2$, (\ref{ikea}) and (\ref{parisfc}), we deduce that $(\widetilde y,\widetilde\omega)\in X_2$ and
$$
\|(\widetilde y,\widetilde\omega)\|_{X_2}\lesssim \|\varsigma_0f_0\|_{L^2(Q)}+\|\varsigma_1f_1\|_{L^2(Q)}+\|(y_0,\omega_0\|_{H^1(\Omega)}.
$$
This concludes the proof of Proposition \ref{prop22}.
\end{proof}

\subsection{Local controllability of the semilinear problem}

In this section we only prove Theorem \ref{Control3} since the proof of Theorem \ref{Control2} is analog and can be derived from what follows.

Our proof relies on the arguments presented in \cite{OlegN-S}. The result of null controllability for the linear system \eqref{linear} given by Proposition~\ref{prop2} will allow us to apply an inverse mapping theorem, which we present now :

\begin{theorem}\label{teo:invmap}
Let $D_1$ and $D_2$ be two Banach spaces and let $\mathcal{A}:D_1 \to D_2$ satisfy $\mathcal{A}\in C^1(D_1;D_2)$. Assume that $x_1\in D_1$, $\mathcal{A}(x_1)=x_2$ and that $\mathcal{A}'(x_1):D_1 \to D_2$ is surjective. Then, there exists $\delta >0$ such that, for every $x'\in D_2$ satisfying $\|x'-x_2\|_{D_2}< \delta$, there exists a solution of the equation
$$\mathcal{A}(x) = x',\quad x\in D_1.$$
\end{theorem}

We apply this theorem for the spaces $D_1 = E_1$ (recall that $E_1$ is defined in $(\ref{E1})$), $D_2 = \rho_0L^2(Q) \times \rho_1 L^2(Q)\times V\times (H^2(\Omega)\cap H^1_0(\Omega))$
and the operator
$$
\mathcal{A}(y,p,u,\omega) = (Ly + (y\cdot \nabla)y +\nabla p -\nabla\times\omega ,
M\omega+(y\cdot\nabla)\overline\omega+(y\cdot\nabla)\omega-\nabla\times y- \mathds{1}_{\cal O}v , y(0,\cdot),\omega(0,\cdot)) 
$$
for $(y,p,u,\omega)\in D_1$.

In order to apply Theorem \ref{teo:invmap}, it remains to check that the operator $\mathcal{A}$ is of class $C^1(D_1;D_2)$. Indeed, notice that all the terms in $\mathcal{A}$ are linear, except for $(y\cdot \nabla)y$ and $(y\cdot\nabla)\omega$. We will prove that the bilinear operators
$$
((y^1,p^1,u^1,\omega^1),(y^2,p^2,u^2,\omega^2))\longmapsto((y^1\cdot \nabla)y^2,(y^1\cdot\nabla)\omega^2)
$$
 are continuous from $D_1\times D_1$ to $ \rho_0 L^2(Q)\times \rho_1L^2(Q)$.
Using the definition of $E_1$ and the fact that
$$
(\rho_0)^{1/2}\leq (\rho_1)^{1/2}\leq (\rho_3)^{3/4}\quad\hbox{and}\quad (\rho_0)^{1/2}\leq (\rho_1)^{1/2}\leq (\rho_0)^{3/4},
$$
we obtain (since $H^1(\Omega) \hookrightarrow L^6(\Omega)$)
\begin{equation*}
\begin{split}
&\|\rho_0(y^1\cdot \nabla)y^2\|_{L^2(Q)}+\|\rho_1(y^1\cdot \nabla)\omega^2\|_{L^2(Q)} \\
&\lesssim \|(\rho_3)^{3/4}y^1\|_{L^\infty(0,T;H^1(\Omega))} \left(\|(\rho_3)^{3/4}\nabla y^2\|_{L^2(0,T;H^1(\Omega))}+\|(\rho_0)^{3/4}\nabla \omega^2\|_{L^2(0,T;H^1(\Omega))}\right)\\
&\lesssim \|(y^1,p^1,u^1,\omega^1)\|_{D_1}\|(y^2,p^2,u^2,\omega^2)\|_{D_1}.
\end{split}
\end{equation*}
Moreover $\mathcal{A}'(0,0,0,0):D_1\to D_2$ is given by
$$\mathcal{A}'(0,0,0,0)(y,p,u, \omega) = (Ly  +\nabla p - \nabla\times \omega ,
M\omega+(y\cdot\nabla)\overline\omega-\nabla\times y- \mathds{1}_{\cal O}v , y(0,\cdot),\omega(0,\cdot)),\,\,\forall (y,p,u, \omega)\in D_1,$$
so this functional is surjective in view of the null controllability result for the linear system \eqref{linear} given by Proposition \ref{prop2}.

We are now able to apply Theorem \ref{teo:invmap} for $x_1=(0,0,0,0)$ and $x_2=(0,0)$. In particular, this gives  the existence of a positive number $\delta$ such that, if $\|(y_0, \omega_0)\|_{H^1(\Omega)\times H^2(\Omega)}\leq \delta$, then we can find a control $v$, such that the associated solution $(y,p,u, \omega)$ to \eqref{system} satisfies $y(T)=0$ and $\omega(T)=\overline\omega(T)$ in $\Omega$.

This concludes the proof of Theorem \ref{Control3}.

\appendix
\section{Standard estimates}

We first present some classical energy estimates for the heat equation and for the Stokes system.

\begin{lemma}\label{estimationM}
Let $w\in L^2(0,T;H^1(\Omega))$ be the solution of the system
\begin{equation}\label{Espanyol}
\left\{
\begin{array}{lllll}
     w_t -A w&=&h+h_0 \qquad &\text{ in } Q,& \\
     w &=& 0 \qquad  &\text{ on } \Sigma,&  \\
     w(0, \cdot)&=&0  \qquad &\text{ on } Q,&
\end{array}
\right.
\end{equation}
where $h,\,h_0\in L^2(0,T;H^{-1}(\Omega))$ and $A$ is either $\Delta$ or $\Delta+\nabla(\nabla \cdot)$.
\begin{enumerate}
\item[$(a)$] Let $h\in L^2(Q)$ and $h_0\equiv 0$. Then,
    $$
    w\in L^2(0,T;H^2(\Omega))\cap H^1(0,T;L^2(\Omega))
    $$
    and there exists some constant $C>0$ independent from $h$ such that
\begin{equation}\label{Chili}
\|w\|_{L^2(0,T;H^2(\Omega))}+\|w\|_{H^1(0,T;L^2(\Omega))} \leq C \|h\|_{L^2(Q)} .
\end{equation}

\item[$(b)$] Let $h \in L^2(0,T;H^2(\Omega))\cap H^1(0,T;L^2(\Omega))$ and $h_0\in L^2(0,T;H^2(\Omega)\cap H^1_0(\Omega))$.
Then, there exists a constant $C>0$ independent from $h$ and $h_0$ such that
$$
\|w\|_{L^2(0,T;H^4(\Omega))\cap H^1(0,T;H^2(\Omega))} \leq C (\|h \|_{L^2(0,T;H^2(\Omega))\cap H^1(0,T;L^2(\Omega))}+\|h_0\|_{L^2(0,T;H^2(\Omega))}).
$$
\end{enumerate}
\end{lemma}
\begin{proof}
We only prove $(b)$ since the proof of $(a)$ is classical (see for instance \cite{Lady}).

For the proof of $(b)$, we write $w=w^1+w^2$ where $w^1$ (respectively $w^2$) is the solution of (\ref{Espanyol}) with $h$ (respectively $h_0$) as right-hand side. We observe that $w^1_{t}$ is the solution of (\ref{Espanyol}) with right-hand side $h_t$ and $A w^2$ is the solution of (\ref{Espanyol}) with right-hand side $A h_0$. Applying $(a)$, we get the desired result.
\end{proof}

\begin{lemma}\label{estimationL}
Let $u\in L^2(0,T;V)$ (together with some $p$) be the solution of the system
\begin{equation}\label{Barcelona}
\left\{
\begin{array}{lllll}
     u_t-\Delta u +\nabla p&=&h_V+h \qquad &\text{ in } Q,&  \\
 \nabla \cdot u &=& 0\qquad &\text{ in } Q,& \\
     u &=& 0 \qquad  &\text{ on } \Sigma,&  \\
     u(0, \cdot)&=&0  \qquad &\text{ on } Q,&
\end{array}
\right.
\end{equation}
where $h,\,h_V\in L^2(0,T;H^{-1}(\Omega))$.

\begin{enumerate}
\item[$(a)$] Let $h \in L^2(Q)$ and $h_V\equiv 0$. Then,
    $$
    u\in L^2(0,T;H^2(\Omega))\cap H^1(0,T;L^2(\Omega))
    $$
    and there exists some constant $C>0$ independent from $h$ such that
$$
\|u\|_{L^2(0,T;H^2(\Omega))}+\|u\|_{H^1(0,T;L^2(\Omega))} \leq C\|h\|_{L^2(Q)}.
$$

\item[$(b)$] Let $h_V \in L^2(0,T;V)$ and $h \in L^2(0,T;H^1(\Omega)) \cap H^1(0,T;H^{-1}(\Omega))$. Then,
    $$
    u\in L^2(0,T;H^3(\Omega))\cap H^1(0,T;H^1(\Omega))
    $$
    and there exists some constant $C>0$ independent from $(h_V,h)$ such that
$$
\|u\|_{L^2(0,T;H^3(\Omega))}+\|u\|_{H^1(0,T;H^1(\Omega))} \leq C \left( \|h_V\|_{L^2(0,T;V)}+\|h\|_{L^2(0,T;H^1(\Omega)) \cap H^1(0,T;H^{-1}(\Omega))} \right).
$$
\item[$(c)$] Let $h\in H^1(0,T;L^2(\Omega))\cap L^2(0,T;H^2(\Omega))$ and $h_{V}\in L^2(0,T;H^2(\Omega)\cap V)$ with $h(\cdot,0)=0$ in $\Omega$. Then,
$$
u\in H^1(0,T;H^2(\Omega))\cap L^2(0,T;H^4(\Omega))
$$
and there exists a constant $C>0$ independent of $h$ and $h_V$ such that
$$
\|u\|_{H^1(0,T;H^2(\Omega))\cap L^2(0,T;H^4(\Omega))} \leq C (\|h\|_{H^1(0,T;L^2(\Omega))\cap L^2(0,T;H^2(\Omega))} +
\|h_V\|_{L^2(0,T;H^2(\Omega))})
$$
\end{enumerate}
\end{lemma}
 \begin{proof}
 Let us first remark that $(a)$, $(b)$ and $(c)$ with $h_V\equiv 0$ are contained in  \cite[Theorem 6, pages 100-101]{Lady}.

We now prove $(b)$ with $h\equiv 0$. Without loss of generality, one may assume that $h_V\in C^{\infty}([0,T];V)$.
In order to simplify the notations, let us denote
$$
A(u,p):=-\Delta u+\nabla p.
$$
Let us multiply the equation in (\ref{Barcelona}) by $A(u_t,p_t)\in
L^2(Q)$, integrate in $\Omega$ and integrate by parts. This yields
$$
\begin{array}{l}\displaystyle
\int_{\Omega}|\nabla
u_t|^2+\frac{1}{2}\frac{d}{dt}\int_{\Omega}|Au|^2
=\int_{\Omega}\nabla u_t\cdot\nabla
h_V.
\end{array}
$$
Here, we have used that $h_V(t,\cdot)$ and $u_t(t,\cdot)$ are elements of $V$ for almost every $t\in (0,T)$. From this
identity, using Young's inequality and thanks to the fact that $(A(u,p))_{|t=0}\equiv 0$ in $\Omega$, we have
that $u\in H^1(0,T;V)$ and
\begin{equation}\label{H31}
\|u\|_{H^1(0,T;H^1(\Omega))}\leq C\|h_V\|_{L^2(0,T;H^1(\Omega))}.
\end{equation}
Now, regarding system (\ref{Barcelona}) as a stationary Stokes system
with right-hand side in $V$ (see, for instance, \cite[Proposition 2.2, page 33]{Temam}), one
deduces $u\in L^2(0,T;H^3(\Omega))$ and concludes the proof of $(b)$.

 We finally prove $(c)$ when $h\equiv 0$. Using that $\Delta u_t + \Delta A(u,p)= \Delta h_V$ in $Q$, we get
 $$
 0=-\int_{\Omega} (\Delta u_t + \Delta A(u,p)- \Delta h_V)A(u_t,p_t).
 $$
 Integrating by parts and noting that $A(u_t,p_t)(t, \cdot)\in V$ for almost every $t \in (0,T)$, we deduce
 $$
 \int_{\Omega}|A(u_t,p_t)|^2+\frac{1}{2}\frac{d}{dt}\int_{\Omega}|\nabla A(u,p)|^2=-\int_{\Omega}\Delta h_V\,A(u_t,p_t)
 $$
 From this, we directly obtain $(c)$.
\end{proof}


Next, we recall a useful lemma related to the Carleman weights.
\begin{lemma}\label{dominationH1}
There exists some positive constants $s_0$ and $C$ such that, for any $u \in L^2(0,T;H^1(\Omega))$ and any $s \geq s_0$,
$$\int_Q e^{-2s\alpha} |u|^2 \leq C \left(s^{-2} \int_Q e^{-2s \alpha} \xi^{-2} |\nabla u|^2+\int_{Q_0} e^{-2s\alpha}|u|^2\right).$$
\end{lemma}
{\bf Proof} (see also \cite[Lemma 3]{CG}){\bf.}
Let us set $v:=  e^{-s \alpha} u$ and $f:=\nabla u \in L^2(0,T; L^2(\Omega))$. Writing $u=e^{s \alpha}  v$, one has
$$ e^{-s \alpha}  f=\nabla v+s v \nabla \alpha$$
so that, after an integration by parts,  since $\nabla \xi=-\nabla \alpha$ and for some constant $C>0$,
\begin{eqnarray*} \int_Q e^{-2s \alpha} \xi^{-2} |f|^2&=&\int_Q \xi^{-2} |\nabla v|^2+s^2 \int_Q \xi^{-2}|\nabla \alpha|^2 v^2-s \int_Q \xi^{-2} (\Delta \alpha)  v^2-2s\int_Q \xi^{-3}|\nabla \alpha|^2 v^2 +\int_{\Sigma} \xi^{-2} (\partial_{\nu} \alpha) v^2\\
&\geq&  \lambda ^2 s^2 \int_Q (1-C (s\xi)^{-1}) |\nabla \eta|^2 v^2-C\lambda s^2 \int_Q (s \xi)^{-1}v^2
\end{eqnarray*}
using that $\partial_\nu \alpha \geq 0$ on $\Sigma$, $\nabla \alpha= - \lambda \nabla \eta \xi$ and $|\Delta \alpha| \lesssim (\lambda^2|\nabla \eta|^2\xi+ \lambda \xi)$.

Moreover, since $|\nabla \eta| >0 $ on the compact set $ \overline{\Omega} \backslash {\Omega_0}$, one additionally gets that, for some $c>0$,
$$
\lambda ^2 s^2 \int_Q (1-C (s\xi)^{-1}) |\nabla \eta|^2 v^2-C\lambda  s^2 \int_Q (s \xi)^{-1}v^2 \geq c  s^2 \left( \int_{Q\backslash Q_0}  v^2- \int_{ Q_0}  v^2\right)
$$
for a choice of $s$ such that $s \gtrsim T^8$. This concludes the proof.
\hspace*{\fill}$\Box$\medskip

Finally, we present some bilinear estimate used for the proof of Proposition \ref{Carleman3}.
\begin{lemma}\label{estimationL2H-1}
There exists some $C>0$ such that, for any $\overline{\omega}\in L^\infty \left(0,T; W^{1,3}(\Omega)\right)\cap H^1\left(0,T; L^{3}(\Omega)\right)$, we have
\begin{itemize}
\item[(a)] for all $u \in L^2(0,T;H^{-1}(\Omega))$,
$$
 \| \overline{\omega}u \|_{L^2(0,T;H^{-1}(\Omega))} \leq C  \| \overline{\omega} \|_{L^\infty \left(0,T; W^{1,3}(\Omega)\right)} \|u\|_{L^2(0,T;H^{-1}(\Omega))};
$$
\item[(b)] for all $u \in L^\infty(0,T;L^{2}(\Omega))$,
$$
\| \overline{\omega}_t u \|_{L^2(0,T;H^{-1}(\Omega))} \leq C  \| \overline{\omega} \|_{H^1 \left(0,T; L^3(\Omega)\right)} \|u\|_{L^\infty(0,T;L^{2}(\Omega))}.
$$
\end{itemize}
\end{lemma}
\begin{proof}
By duality, the first estimate reduces to prove that
\begin{equation} \label{pot}\forall u \in  L^2(0,T;H^{1}_0(\Omega)), \ \| \overline{\omega}u \|_{L^2(0,T;H^{1}_0(\Omega))} \lesssim  \| \overline{\omega} \|_{L^\infty \left(0,T; W^{1,3}(\Omega)\right)} \|u\|_{L^2(0,T;H^{1}_0(\Omega))}.
\end{equation}
Morover, one has
\begin{eqnarray*}\forall u \in  L^2(0,T;H^{1}_0(\Omega)), \ \| \overline{\omega}u \|_{L^2(0,T;H^{1}_0(\Omega))} &\lesssim&   \| \overline{\omega}\nabla u \|_{L^2(0,T;L^2(\Omega))}+ \|(\nabla \overline{\omega})u \|_{L^2(0,T;L^2(\Omega))}\\
 & \lesssim & \| \overline{\omega} \|_{L^\infty \left(0,T; L^\infty(\Omega)\right)}\|u\|_{L^2(0,T;H^{1}_0(\Omega))}\\
 &+& \| \nabla \overline{\omega} \|_{L^\infty \left(0,T; L^{3}(\Omega)\right)}\|u\|_{L^2(0,T;H^{1}_0(\Omega))}
\end{eqnarray*}
since $H^1(\Omega) \hookrightarrow L^6(\Omega)$. This concludes the proof of (a).

On the other hand, one has by duality  $L^{6/5}(\Omega) \hookrightarrow H^{-1}(\Omega)$ so that
$$\forall u \in L^\infty(0,T;L^{2}(\Omega)), \ \| \overline{\omega}_t u \|_{L^2(0,T;H^{-1}(\Omega))} \lesssim  \| \overline{\omega}_t u \|_{L^2 \left(0,T; L^{6/5}(\Omega)\right)} \lesssim \| \overline{\omega}_t \|_{L^2 \left(0,T; L^{3}(\Omega)\right)} \| u \|_{L^\infty \left(0,T; L^2(\Omega)\right)}   $$
using H\"{o}lder inequality. The proof of (b) is complete.
\end{proof}

\end{document}